\documentclass[11pt]{amsart}

\usepackage{amssymb,caption,enumitem,mathtools,tikz,tikz-cd}
\usepackage[colorlinks,allcolors=blue]{hyperref}
\usepackage[margin=1.25in]{geometry}

\usetikzlibrary{calc}

\tikzset{main_node/.style={circle,draw=black,fill=black,minimum size=5.5pt,inner sep=0pt]}}
\tikzset{every picture/.style={line width=0.7pt}}
\tikzset{highlight/.style={line width=2pt,red}}

\makeatletter
\g@addto@macro\bfseries{\boldmath}
\makeatother

\newtheorem{theorem}{Theorem}
\newtheorem{corollary}[theorem]{Corollary}
\newtheorem{lemma}[theorem]{Lemma}
\newtheorem{proposition}[theorem]{Proposition}

\theoremstyle{remark}
\newtheorem{example}[theorem]{Example}
\newtheorem{remark}[theorem]{Remark}

\numberwithin{theorem}{section}

\DeclareMathOperator{\diag}{diag}
\DeclareMathOperator{\im}{im}
\DeclareMathOperator{\Star}{Star}

\DeclarePairedDelimiter{\abs}{\lvert}{\rvert}

\newcommand{\vd}{\mathbf{d}}
\newcommand{\vr}{\mathbf{r}}
\newcommand{\ZZ}{\mathbb{Z}}

\begin{document}

\title[Generalized chip firing and critical groups associated to trees]{Generalized chip firing and critical groups\\ of arithmetical structures on trees}

\author{Kassie Archer}
\address[K.~Archer]{Department of Mathematics\\ United States Naval Academy\\ Annapolis, MD 21402\\ USA \newline Disclaimer: The views expressed in this paper are those of the authors and do not reflect the official policy or position of the U.S. Naval Academy, Department of the Navy, the Department of Defense, or the U.S. Government.}
\email{karcher@usna.edu}

\author{Alexander Diaz-Lopez}
\address[A.~Diaz-Lopez]{Department of Mathematics and Statistics\\ Villanova University\\ 800 Lancaster Ave (SAC 305)\\ Villanova, PA 19085\\ USA}
\email{alexander.diaz-lopez@villanova.edu}

\author{Darren Glass}
\address[D.~Glass]{Department of Analytics\\ Dickinson College\\ 28 N. College St.\\ Carlisle, PA 17013\\ USA}
\email{glassd@dickinson.edu}

\author{Joel Louwsma}
\address[J.~Louwsma]{Department of Mathematics\\ Niagara University\\ Niagara University, NY 14109\\ USA}
\email{jlouwsma@niagara.edu}

\begin{abstract}
Chip firing provides a way to study the sandpile group (also known as the Jacobian) of a graph. We use a generalized version of chip firing to bound the number of invariant factors of the critical group of an arithmetical structure on a graph. We also show that, under suitable hypotheses, critical groups are additive under wedge sums of graphs with arithmetical structures. These results allow us to relate the number of invariant factors of critical groups associated to any given tree to decompositions of the tree into simpler trees. We use this to classify those trees for which every arithmetical structure has cyclic critical group. Finally, we show how to construct arithmetical structures on trees with prescribed critical groups. In particular, every finite abelian group is realized as the critical group of some arithmetical structure on a tree. 
\end{abstract}

\maketitle

\section{Introduction}

Let $G$ be a finite, connected graph without loops, and let $A(G) = (a_{v,w})_{v,w \in V(G)}$ be the adjacency matrix of~$G$. An \emph{arithmetical structure} on~$G$ is a pair $(\vd,\vr)$ of vectors indexed by the vertices of~$G$ such that $\vd = (\vd(v))_{v \in V(G)}$ has nonnegative integer entries, $\vr = (\vr(v))_{v \in V(G)}$ has positive integer entries with no nontrivial common factor in~$\ZZ$, and
\begin{equation}\label{eq:arithmeticalstructure}
(\diag(\vd)-A(G))\vr = \mathbf{0}.
\end{equation}
Equivalently, arithmetical structures can be defined as positive integer labels $\{\vr(v)\}_{v \in V(G)}$ on the vertices of~$G$ such that for each vertex~$v$ we have 
\[\vd(v)\vr(v) = \sum_{w \in V(G)}a_{v,w}\vr(w)\] 
for some $\vd(v) \in \ZZ_{\geq 0}$. In other words, at each vertex, the $\vr$-label there divides the sum of the $\vr$-labels at adjacent vertices, counted with multiplicity. Arithmetical structures were originally introduced by Lorenzini~\cite{L89} to study intersections of degenerating curves in algebraic geometry, but in recent years they have also been studied as combinatorial objects. 

The matrix $\diag(\vd)-A(G)$ in Equation~\eqref{eq:arithmeticalstructure} is called the \emph{generalized Laplacian matrix} of the arithmetical structure and is denoted $L(G,\vd)$. This matrix has rank $\abs{V(G)}-1$, as shown in \cite[Proposition~1.1]{L89}. To each arithmetical structure, there is an associated \emph{critical group} $\mathcal{K}(G;\vd,\vr)$ that is isomorphic to the torsion part of the cokernel of $L(G,\vd)$, i.e.,
\[\ZZ^{V(G)}/\im(L(G,\vd)) \cong \ZZ \oplus \mathcal{K}(G;\vd,\vr).\]
In the special case when $\vr(v)=1$ and $\vd(v)=\deg(v)$ for all $v \in V(G)$, the arithmetical structure $(\vd,\vr)$ is called the \emph{Laplacian arithmetical structure}. Many authors (e.g., Biggs~\cite{Biggs}) use the term \emph{critical group of a graph} to refer specifically to the critical group of the Laplacian arithmetical structure; this group also has several other names in the literature, including the sandpile group~\cite{Dhar}, the Jacobian~\cite{BN}, and the Picard group~\cite{BHN}. For more details about these notions, see~\cite{CP} and the references therein. 

The critical group of an arithmetical structure can be computed via the Smith normal form of its generalized Laplacian matrix. Specifically, $\mathcal{K}(G;\vd,\vr) \cong \bigoplus_{i=1}^{n-1} \ZZ/\alpha_i\ZZ$, where $\alpha_i$ is the $i$-th diagonal entry of the Smith normal form of $L(G,\vd)$. It follows from results of Raynaud~\cite{Raynaud} that the critical group is isomorphic to the group of components of the N\'{e}ron model of the Jacobian associated to the generic curve. Lorenzini further discusses the geometric motivation for the study of critical groups of arithmetical structures in~\cite{L90}.

For a tree~$T$, a result of Lorenzini \cite[Corollary~2.5]{L89} gives the order of the critical group of an arithmetical structure in terms of~$\vr$ and the degrees of the vertices. Specifically, 
\[\abs{\mathcal{K}(T;\vd,\vr)} = \prod_{v \in V(T)} \vr(v)^{\deg(v)-2}.\]
However, the overall structure of the critical group is generally not clear, and there are many open questions in this direction. For example, among graphs with a fixed number of vertices, which groups arise as the critical group of an arithmetical structure on one of these graphs? What about for graphs with a fixed number of edges? For a given finite abelian group~$\mathcal{G}$, what is the smallest graph that admits an arithmetical structure with $\mathcal{G}$ as critical group? Which graphs only admit arithmetical structures with cyclic critical groups? There are some partial results about these questions for certain families of graphs.

In~\cite{B18}, Braun et al.\ show that for path graphs $\mathcal{K}(P_n;\vd,\vr)$ is always trivial, while for cycle graphs $\mathcal{K}(C_n;\vd,\vr)$ is a cyclic group whose order is the number of entries of~$\vr$ that are~$1$. In \cite[Proposition~5.1]{A20}, Archer et al.\ show that critical groups of arithmetical structures on bidents are always cyclic, while in \cite[Theorem~1]{A24} critical groups of arithmetical structures on star graphs are determined in terms of the $\vd$-values at the leaves, which also follows from \cite[Theorem~2.1]{L92}. In \cite[Corollary~2.10]{L24}, Lorenzini shows that every graph admits an arithmetical structure with trivial critical group.

This paper uses a generalized notion of chip firing on arithmetical structures to study critical groups of arithmetical structures on trees. We begin in Section~\ref{sec:divisorschipfiring}, where we interpret elements of critical groups of arithmetical structures on general graphs as equivalence classes of divisors under chip-firing relations. The main result of this section, Theorem~\ref{thm:invariantfactorbound}, bounds the number of invariant factors of a critical group in terms of the number of vertices at which divisor representatives can be made to have no chips. In Section~\ref{sec:merging}, we study the behavior of critical groups under wedge sums of graphs with arithmetical structures. Theorem~\ref{thm:mergingcriticalgroup} shows that, when the $\vr$-values at the vertices to be merged are relatively prime, the critical group is additive under this operation, extending a result of Lorenzini \cite[Proposition~4.3]{L00}. 

Beginning in Section~\ref{sec:boundinginvariantfactors}, we turn our attention to trees, decomposing them into pieces that have at most one vertex of degree greater than~$2$. Using the chip-firing perspective from Section~\ref{sec:divisorschipfiring} and the merging perspective from Section~\ref{sec:merging}, we bound the number of invariant factors of a critical group in terms of the number of leaves of these pieces (Theorem~\ref{thm:invariantfactorsboundtrees}). As a corollary, we classify those trees for which every arithmetical structure has cyclic critical group. In Section~\ref{sec:invariants}, we introduce an invariant we call the splitting irregularity number of a tree and relate it to the $2$-matching number. Theorem~\ref{thm:invariantfactorsboundtrees2} then reinterprets Theorem~\ref{thm:invariantfactorsboundtrees} in terms of the splitting irregularity number and $2$-matching number, recovering and extending a result of Corrales and Valencia \cite[Theorem~3.9]{CV15}. 

Finally, in Section~\ref{sec:constructgroups}, we construct arithmetical structures on trees with prescribed critical groups. Every finite abelian group is realized as the critical group of some arithmetical structure on a tree. Furthermore, Corollary~\ref{cor:constructgroups} shows that, if $T$ is a tree, $\ell(T)$ is its number of leaves, and $\mathcal{G}$ is a finite abelian group with at most $\ell(T)-2$ invariant factors, then there is an arithmetical structure on a subdivision of~$T$ with critical group~$\mathcal{G}$. A more general statement involving the splitting irregularity numbers of the trees is given in Theorem~\ref{thm:constructgroups}.

\section{Divisors, chip firing, and the critical group}\label{sec:divisorschipfiring}

Given a finite, connected graph~$G$ without loops, there is a well-known chip-firing game (see \cite{CP,Klivans} for an overview) in which one considers the set of divisors (labelings of the vertices by integers) under an equivalence relation involving \emph{firing} and \emph{borrowing} at vertices. Replacing the standard firing and borrowing by a variation involving~$\vd$, one obtains a notion of chip firing on an arithmetical structure on~$G$, as described in \cite[Section~4]{AB}. In the case of the Laplacian arithmetical structure, the set of divisors with degree (sum of the labels)~$0$, up to firing and borrowing, is isomorphic to the critical group of the graph. Generalizing the notion of degree to a weighted sum involving~$\vr$, this result extends to any arithmetical structure on~$G$. 

For completeness, we define the generalized chip-firing game that we consider. Denote the adjacency matrix of~$G$ by $A(G) = (a_{v,w})_{v,w \in V(G)}$. A \emph{divisor}~$\delta$ on~$G$ is a function $\delta \colon V(G) \to \ZZ$; the integer value at each vertex can be thought of as giving the number of chips there. For a fixed $\vd \in \ZZ^{V(G)}$, we can obtain an equivalence relation on the set of divisors by defining firing and borrowing as follows. Given a divisor~$\delta$, \emph{firing} at vertex~$v$ adds $a_{v,w}$ chips to each vertex~$w$ other than~$v$ and removes $\vd(v)$ chips from~$v$. Similarly, \emph{borrowing} at vertex~$v$ (or, equivalently, firing $-1$~times at~$v$) removes $a_{v,w}$ chips from each vertex~$w$ other than~$v$ and adds $\vd(v)$ chips to~$v$. Letting $\delta'$ be the divisor that results from~$\delta$ after firing at~$v$ and $\delta''$ be the divisor that results from~$\delta$ after borrowing at~$v$, we have
\[\delta'(w) = \begin{cases}
 \delta(w) + a_{v,w} & \text{for } w \neq v, \\
 \delta(w) - \vd(w) & \text{for } w=v
\end{cases}\qquad\text{and}\qquad
\delta''(w) = \begin{cases}
 \delta(w) - a_{v,w} & \text{for } w \neq v, \\
 \delta(w) + \vd(w) & \text{for } w=v.
\end{cases}\]
Thinking of a divisor~$\delta$ as a vector in $\ZZ^{V(G)}$, firing or borrowing at~$v$ is equivalent to subtracting or adding, respectively, the column of the matrix $\diag(\vd)-A(G)$ indexed by~$v$. Therefore, two divisors are equivalent exactly when they differ by an element in the image of $\diag(\vd)-A(G)$, considered as a map $\ZZ^{V(G)} \to \ZZ^{V(G)}$. 

In this paper, we will primarily be concerned with the case when $\vd$ determines an arithmetical structure $(\vd,\vr)$ on~$G$. In this case, we can define the \emph{degree} of a divisor~$\delta$ with respect to~$\vr$, denoted $\deg_{\vr}(\delta)$, to be the sum of the labels weighted by the $\vr$-values, i.e., $\sum_{v \in V(G)} \vr(v)\delta(v)$. Note that the degree of a divisor is preserved under the firing and borrowing operations. The set of equivalence classes of divisors forms a group isomorphic to $\ZZ^{V(G)}/\im(L(G,\vd))$, the full cokernel of $L(G,\vd)$; the set of equivalence classes of divisors of degree~$0$ with respect to~$\vr$ is isomorphic to the torsion part of the cokernel, i.e., the critical group $\mathcal{K}(G;\vd,\vr)$.

Let $\delta$ be a divisor on~$G$, and let $\mathbf{x} \in \ZZ^{V(G)}$. Starting with~$\delta$ and firing $\mathbf{x}(v)$ times at~$v$ for all $v \in V(G)$ is the same as subtracting $(\diag(\vd)-A(G))\mathbf{x}$ from~$\delta$. In the case of an arithmetical structure $(\vd,\vr)$, since $\vr$ is in the kernel of $L(G,\vd)$, firing $\vr(v)$ times at~$v$ for all $v \in V(G)$ does not change a divisor.

\begin{example}
Let $C_4$ be the cycle graph on four vertices with $V(C_4) = \{v_1,v_2,v_3,v_4\}$, where $a_{v_i,v_j}$ is equal to~$0$ if $\{i,j\}$ is $\{1,3\}$ or $\{2,4\}$ and equal to~$1$ otherwise. An arithmetical structure on~$C_4$ is given by $(\vd(v_i))_{i=1}^4 = (3,2,1,4)$ and $(\vr(v_i))_{i=1}^4 = (1,2,3,1)$, with associated generalized Laplacian matrix
\[L(C_4,\vd) = \begin{bmatrix}
 3 & -1 & 0 & -1 \\ 
 -1 & 2 & -1 & 0 \\ 
 0 & -1 & 1 & -1 \\ 
 -1 & 0 & -1 & 4
\end{bmatrix}.\]
Consider the divisor $\delta = (3,1,-1,-2)$. The degree of this divisor is $\deg_{\vr}(\delta) = 1 \cdot 3 + 2 \cdot 1 + 3 \cdot (-1) + 1 \cdot (-2) = 0$. By firing once at vertex~$v_2$, we obtain the divisor $\delta' = (4,-1,0,-2)$. If we instead fire at~$v_1$ three times, borrow at~$v_3$ five times, and fire at~$v_4$ one time, this is equivalent to subtracting from~$\delta$ the vector $L(G,\vd)\mathbf{v}$, where $\mathbf{v} = (3,0,-5,1)^T$, which gives the new divisor $\delta'' = (-5,-1,5,-8)$. 
 
Although $\delta$ has degree~$0$, it is not equivalent to the divisor~$\epsilon$ with $\epsilon(v)=0$ for all $v \in V(G)$, which also has degree~$0$. If it were, then $\delta$ would be in the image of $L(G,\vd)$ as a map $\ZZ^{V(G)} \to \ZZ^{V(G)}$, but it is straightforward to check that it is not. However, $2\delta$ is in the image of $L(G,\vd)$ since $2\delta = L(G,\vd)\mathbf{w}$ for $\mathbf{w} = (3,3,1,0)^T$, and thus $\delta$ has order~$2$ in the critical group $\mathcal{K}(C_4;\vd,\vr)$. 
\end{example}

A \emph{tentacle} of a graph is a collection $\{v_0, v_1, \dots, v_k\}$ of vertices with $k \geq 1$, where $v_i \sim v_{i+1}$ for $0 \leq i \leq k-1$, $\deg(v_i)=2$ for all $i \in [k-1]$, $\deg(v_k)=1$, and $\deg(v_0) \neq 2$. We call $v_0$ the \emph{base} of the tentacle and $k$ the \emph{length} of the tentacle.

\begin{remark}\label{rem:borrowalongtentacle}
If a graph~$G$ has a tentacle $\{v_0, v_1, \dots, v_k\}$, then any divisor~$\delta$ on~$G$ is equivalent with respect to any $\vd \in \ZZ^{V(G)}$ to a divisor~$\delta'$ that is $0$ at $v_1, v_2, \dots, v_k$. Indeed, one can find such a~$\delta'$ by first borrowing at $v_{k-1}$ to obtain a divisor that is $0$ at~$v_k$, then borrowing at $v_{k-2}$ to obtain a divisor that is also $0$ at $v_{k-1}$, and continuing in this way through borrowing at~$v_0$. Similarly, $\delta$ is equivalent to a divisor~$\delta''$ that is $0$ at $v_0, v_1, \dots, v_{k-1}$. We can find such a~$\delta''$ by first borrowing at~$v_1$ to obtain a divisor that is $0$ at~$v_0$, then borrowing at~$v_2$ to obtain a divisor that is also $0$ at~$v_1$, and continuing in this way through borrowing at~$v_k$. We will use this process in Section~\ref{sec:boundinginvariantfactors} when working with divisors on trees. 
\end{remark}

In the next proposition, we relate certain properties of divisors on~$G$ to properties of the matrix $\diag(\vd)-A(G)$. 

\begin{proposition}\label{prop:equivalent}
Let $G$ be a finite, connected graph without loops, and let $\vd \in \ZZ^{V(G)}$. Fix subsets $X,Y \subseteq V(G)$ with $\abs{X} \leq \abs{Y}$, and let $M$ be the submatrix of $\diag(\vd)-A(G)$ determined by rows~$X$ and columns~$Y$. The following are equivalent:
\begin{enumerate}[label=\textup{(\alph*)},ref=\textup{\alph*}]
\item Every divisor on~$G$ is equivalent under firing \textup{(}with respect to $\vd$\textup{)} at vertices in~$Y$ to a divisor that is $0$ at vertices in~$X$.\label{item:clearchips} 
\item The matrix~$M$ has a right inverse over~$\ZZ$. \label{item:rightinverse}
\item The greatest common divisor of the $\abs{X} \times \abs{X}$ minors of~$M$ is~$1$. \label{item:gcd1}
\end{enumerate}
\end{proposition}

\begin{proof}
We first prove that (\ref{item:clearchips}) implies~(\ref{item:rightinverse}). Denote the columns of~$M$ by $\{\textbf{m}_y\}_{y \in Y}$. Suppose any divisor is equivalent under firing at vertices in~$Y$ to a divisor with no chips at vertices in~$X$. This means that, given any $\mathbf{v} \in \ZZ^{X}$, we can find integers $\{a_y\}_{y \in Y}$ such that 
\[\mathbf{v} - \sum_{y \in Y} a_y\textbf{m}_y = \textbf{0}.\] 
Letting $\mathbf{v}$ be each of the columns of the identity matrix~$I_X$, the corresponding values of~$a_y$ produce the columns of a matrix that is a right inverse of~$M$.

We next prove that (\ref{item:rightinverse}) implies~(\ref{item:clearchips}). If $M$ has a right inverse, then each of the columns of~$I_X$ is an integral linear combination of the columns of~$M$, and hence every vector in~$\ZZ^{X}$ is a linear combination of the columns of~$M$. Given a divisor, let $\mathbf{v}$ record the number of chips at each vertex. The coefficients of the linear combination that gives $-\mathbf{v}$ as a combination of the columns of~$M$ tell us how many times to fire at each of the vertices in~$Y$ so there are no chips at vertices in~$X$. 

We next prove that (\ref{item:rightinverse}) implies~(\ref{item:gcd1}). Suppose that $M$ has a right inverse over~$\ZZ$; call it~$N$. Applying the Cauchy--Binet formula to $MN=I_{X}$, we get that 
\[\sum_{S \in \binom{V(G)}{\abs{X}}}\det(M_{X,S})\det(N_{S,X}) = 1.\]
This says that an integral linear combination of all $\abs{X} \times \abs{X}$ minors of~$M$ is~$1$, so hence the greatest common divisor of these minors is~$1$.

Finally, we prove that (\ref{item:gcd1}) implies~(\ref{item:rightinverse}). If the greatest common divisor of the $\abs{X} \times \abs{X}$ minors of~$M$ is~$1$, then from the Smith normal form of~$M$ we have that $UMV=S$, where $U$ and~$V$ are invertible over~$\ZZ$ and $S$ is an $\abs{X} \times \abs{Y}$ matrix that is $I_{X}$ with $\abs{Y}-\abs{X}$ columns of zeros appended. Restricting to the columns indexed by~$X$, we have that $UMV'=I_{X}$, where $V'$ consists of the columns of~$V$ indexed by~$X$. Multiplying by $U^{-1}$ on the left and $U$ on the right, we obtain $MV'U=I_{X}$, so hence $M$ has a right inverse.
\end{proof}

When $\vd$ determines an arithmetical structure, Proposition~\ref{prop:equivalent} allows us to give an upper bound on the number of invariant factors of the critical group of that arithmetical structure. 

\begin{theorem}\label{thm:invariantfactorbound}
Let $G$ be a finite, connected graph without loops. Fix a nonempty subset $Z \subseteq V(G)$ and an arithmetical structure $(\vd,\vr)$ on~$G$. If every divisor on~$G$ is equivalent \textup{(}with respect to~$\vd$\textup{)} to a divisor that is nonzero only at vertices in~$Z$, then $\mathcal{K}(G;\vd,\vr)$ has at most $\abs{Z}-1$ invariant factors.
\end{theorem}

\begin{proof}
The hypothesis says that condition~(\ref{item:clearchips}) of Proposition~\ref{prop:equivalent} is satisfied with $X = V(G) \setminus Z$ and $Y=V(G)$. Condition~(\ref{item:gcd1}) of Proposition~\ref{prop:equivalent} then implies that the first~$\abs{X}$ diagonal entries of the Smith normal form of $L(G,\vd)$ are~$1$, and the fact that we have an arithmetical structure means that the last diagonal entry of the Smith normal form is~$0$. Since there are only $\abs{V(G)}-\abs{X}-1 = \abs{Z}-1$ other diagonal entries of the Smith normal form, the result holds.
\end{proof}

As an example, we can use Theorem~\ref{thm:invariantfactorbound} to recover the result from \cite[Theorem~7]{B18} that critical groups of arithmetical structures on path graphs are always trivial.

\begin{example}\label{ex:pathgraph}
For a path graph, any divisor is equivalent with respect to any $\vd \in \ZZ^{V(G)}$ to one with all chips at one vertex, by using the procedure described in Remark~\ref{rem:borrowalongtentacle}. The conditions of Theorem~\ref{thm:invariantfactorbound} are thus satisfied for some~$Z$ of cardinality~$1$, so the critical group of every arithmetical structure is trivial. 
\end{example}

The bound given by Theorem~\ref{thm:invariantfactorbound} is not necessarily sharp, as the next example shows.

\begin{example}
Consider the graph obtained by merging a $3$-cycle, a $5$-cycle, and a $7$-cycle at one vertex with the Laplacian arithmetical structure. The critical group is the cyclic group $\ZZ/105\ZZ$ and so has one invariant factor. Becker and Glass \cite[Example~2.3]{BG} show that any element of the critical group that is represented by a divisor that is nonzero at only two vertices cannot have order $105$. Therefore not all divisors are equivalent to a divisor that is nonzero at only two vertices, so Theorem~\ref{thm:invariantfactorbound} cannot be used to show that the critical group has at most one invariant factor. 

The element of the critical group represented by the divisor~$\delta$ obtained by putting $1$~chip on a vertex of the $3$-cycle, $1$~chip on a vertex of the $5$-cycle and $-2$~chips on a vertex of the $7$-cycle has order $105$, implying that every divisor is equivalent to one with chips at only these vertices. Letting $Z$ be the set of these three vertices, Theorem~\ref{thm:invariantfactorbound} gives the non-sharp bound that the critical group has at most $\abs{Z}-1=2$ invariant factors.
\end{example}

\section{Merging arithmetical structures}\label{sec:merging}

Suppose $G_1$ and~$G_2$ are two finite, connected graphs without loops, and choose $x \in V(G_1)$ and $y \in V(G_2)$. Let $G_1 \vee G_2$ be the graph obtained from $G_1$ and~$G_2$ by identifying the vertices $x$ and~$y$. In this section, we describe how to obtain an arithmetical structure on $G_1 \vee \nolinebreak G_2$ from arithmetical structures on $G_1$ and~$G_2$. In the case when the $\vr$-values at $x$ and~$y$ are relatively prime, we will show that the critical group of the resulting arithmetical structure is isomorphic to the direct sum of the critical groups of the two original arithmetical structures. 

We first describe how to obtain an arithmetical structure on $G_1 \vee G_2$ from arithmetical structures on $G_1$ and~$G_2$. This merging operation was previously studied by Hower in his thesis \cite[Definition~7.1]{Hower}. It is also closely related to operations studied by Corrales and Valencia \cite{CV18-LAA,CV18-JAA}. 

\begin{proposition}\label{prop:mergingconstruction}
Let $G_1$ and~$G_2$ be finite, connected graphs without loops, and choose $x \in V(G_1)$ and $y \in V(G_2)$. If $(\vd_1,\vr_1)$ and $(\vd_2,\vr_2)$ are arithmetical structures on $G_1$ and~$G_2$, respectively, then there is an arithmetical structure on $G_1 \vee G_2$ defined by
\[\vd(v) = \begin{cases}
 \vd_1(v) & \text{if } v \in V(G_1)\setminus\{x\}, \\
 \vd_2(v) & \text{if }v \in V(G_2)\setminus\{y\}, \\
 \vd_1(x) + \vd_2(y) & \text{if }v \text{ is the vertex of } G_1 \vee G_2 \text{ resulting from merging $x$ and~$y$}
\end{cases}\]
and
\[\vr(v) = \begin{cases}
 \frac{b}{\gcd(a,b)} \vr_1(v) & \text{for } v \in V(G_1), \\
 \frac{a}{\gcd(a,b)} \vr_2(v) & \text{for }v \in V(G_2),
\end{cases}\]
where $a=\vr_1(x)$ and $b=\vr_2(y)$.
\end{proposition}

\begin{proof}
Note that $\vr$ is well defined as the given values agree at the merged vertex. The fact that $a/\gcd(a,b)$ and $b/\gcd(a,b)$ are relatively prime implies that the entries of~$\vr$ have no nontrivial common factor. One can check that $(\diag(\vd)-A)\vr = \mathbf{0}$ since $(\vd_1,\vr_1)$ and $(\vd_2,\vr_2)$ are arithmetical structures on $G_1$ and~$G_2$, as done in \cite[Theorem~2.3]{CV18-JAA}. Thus $(\vd,\vr)$ is an arithmetical structure on $G_1 \vee G_2$.
\end{proof}

The following example illustrates this construction.

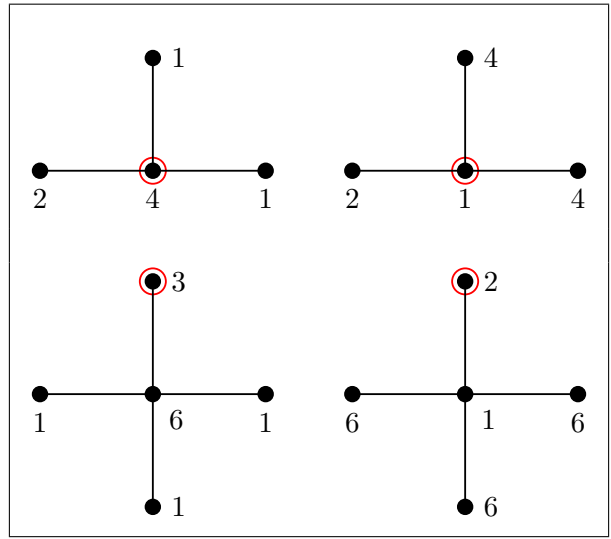
\begin{figure}
\centering
\begin{tikzpicture}[scale=.77]

\draw[thick] (-3,-3) -- (-3,7) -- (8.5,7) -- (8.5,-3) -- (-3,-3);

\draw[red, fill=white] (0, 4) circle[radius=0.6em]; 
\node[main_node, label=below:{$1$}] (C1) at (0, 4) {};

\node[main_node, label=right:{$4$}] (N1) at (0, 6) {};
\node[main_node, label=below:{$4$}] (E1) at (2, 4) {};
\node[main_node, label=below:{$2$}] (W1) at (-2, 4) {};

\draw (C1.center) -- (N1.center);
\draw (C1.center) -- (E1.center);
\draw (C1.center) -- (W1.center);

\draw[red, fill=white] (5.5, 4) circle[radius=0.6em]; 
\node[main_node, label=below:{$4$}] (C3) at (5.5, 4) {};

\node[main_node, label=right:{$1$}] (N3) at (5.5, 6) {};
\node[main_node, label=below:{$1$}] (E3) at (7.5, 4) {};
\node[main_node, label=below:{$2$}] (W3) at (3.5, 4) {};

\draw (C3.center) -- (N3.center);
\draw (C3.center) -- (E3.center);
\draw (C3.center) -- (W3.center);

\node[main_node, label=below right:{$1$}](C2) at (0, 0) {};

\draw[red, fill=white] (0, 2) circle[radius= 0.6 em]; 
\node[main_node, label=right:{$2$}] (N2) at (0, 2) {};
\node[main_node, label=below:{$6$}] (W2) at (-2, 0) {};
\node[main_node, label=right:{$6$}] (S2) at (0, -2) {};
\node[main_node, label=below:{$6$}] (E2) at (2, 0) {};

\draw (C2.center) -- (N2.center);
\draw (C2.center) -- (W2.center);
\draw (C2.center) -- (S2.center);
\draw (C2.center) -- (E2.center);

\node[main_node, label=below right:{$6$}] (C4) at (5.5, 0) {};

\draw[red, fill=white] (5.5, 2) circle[radius= 0.6 em]; 
\node[main_node, label=right:{$3$}] (N4) at (5.5, 2) {};
\node[main_node, label=below:{$1$}] (W4) at (3.5, 0) {};
\node[main_node, label=right:{$1$}] (S4) at (5.5, -2) {};
\node[main_node, label=below:{$1$}] (E4) at (7.5, 0) {};

\draw (C4.center) -- (N4.center);
\draw (C4.center) -- (W4.center);
\draw (C4.center) -- (S4.center);
\draw (C4.center) -- (E4.center);

\end{tikzpicture}

\caption{Arithmetical structures on $G_1$ (upper) and $G_2$ (lower), where the $\vd$-values are given on the left and the $\vr$-values on the right. The arithmetical structure on $G_1 \vee G_2$ obtained by merging $G_1$ and~$G_2$ at the circled vertices is shown in Figure~\ref{fig:postmerge}.}
\label{fig:premerge}
\end{figure}

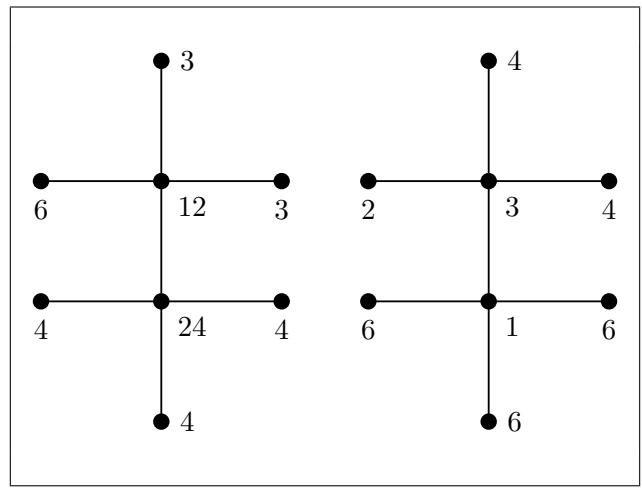
\begin{figure}
\centering
\begin{tikzpicture}[scale=.77]

\draw[thick] (-3,-3) -- (-3,5) -- (8.5,5) -- (8.5,-3) -- (-3,-3);

\node[main_node, label=below right:{$3$}] (C1) at (0, 2) {};

\node[main_node, label=right:{$4$}] (N1) at (0, 4) {};
\node[main_node, label=below:{$2$}] (W1) at (-2, 2) {};
\node[main_node, label=below:{$4$}] (E1) at (2, 2) {};

\draw (C1.center) -- (N1.center);
\draw (C1.center) -- (W1.center);
\draw (C1.center) -- (E1.center);

\node[main_node, label=below right:{$1$}] (C2) at (0, 0) {};

\node[main_node, label=right:{$6$}] (S2) at (0, -2) {};
\node[main_node, label=below:{$6$}] (W2) at (-2, 0) {};
\node[main_node, label=below:{$6$}] (E2) at (2, 0) {};

\draw (C2.center) -- (W2.center);
\draw (C2.center) -- (S2.center);
\draw (C2.center) -- (E2.center);
\draw (C1.center) -- (C2.center);

\node[main_node, label=below right:{$12$}] (C3) at (5.5, 2) {};

\node[main_node, label=right:{$3$}] (N3) at (5.5, 4) {};
\node[main_node, label=below:{$6$}] (W3) at (3.5, 2) {};
\node[main_node, label=below:{$3$}] (E3) at (7.5, 2) {};

\draw (C3.center) -- (N3.center);
\draw (C3.center) -- (W3.center);
\draw (C3.center) -- (E3.center);

\node[main_node, label=below right:{$24$}] (C4) at (5.5, 0) {};

\node[main_node, label=right:{$4$}] (S4) at (5.5, -2) {};
\node[main_node, label=below:{$4$}] (W4) at (3.5, 0) {};
\node[main_node, label=below:{$4$}] (E4) at (7.5, 0) {};

\draw (C4.center) -- (S4.center);
\draw (C4.center) -- (W4.center);
\draw (C4.center) -- (E4.center);
\draw (C3.center) -- (C4.center);

\end{tikzpicture}

\caption{The graph $G_1 \vee G_2$, with $G_1$ and~$G_2$ from Figure \ref{fig:premerge}, together with the arithmetical structure given by Proposition~\ref{prop:mergingconstruction}. The $\vd$-values are given on the left and the $\vr$-values on the right.\label{fig:postmerge}}
\end{figure}

\begin{example}
Let $\Star_3$ and $\Star_4$ be the star graphs with three and four leaves, respectively, with distinguished vertices as indicated in Figure~\ref{fig:premerge}, and let $(\vd_1,\vr_1)$ and $(\vd_2,\vr_2)$ be the arithmetical structures shown in Figure~\ref{fig:premerge}. It follows from \cite[Theorem~1]{A24} that $\mathcal{K}(\Star_3;\vd_1,\vr_1) \cong \ZZ/2\ZZ$ and $\mathcal{K}(\Star_4;\vd_2,\vr_2) \cong \ZZ/2\ZZ \oplus \ZZ/6\ZZ$. The graph $G = \Star_3 \vee \Star_4$ obtained by merging the central vertex of $\Star_3$ and one of the leaves of $\Star_4$, with the arithmetical structure $(\vd,\vr)$ given by Proposition~\ref{prop:mergingconstruction}, is shown in Figure~\ref{fig:postmerge}. One can check that the critical group of this arithmetical structure on~$G$ is 
\[\mathcal{K}(G;\vd,\vr) \cong \ZZ/2\ZZ \oplus \ZZ/2\ZZ \oplus \ZZ/6\ZZ \cong \mathcal{K}(\Star_3;\vd_1,\vr_1) \oplus \mathcal{K}(\Star_4;\vd_2,\vr_2).\] 
\end{example}

When $G = G_1 \vee G_2$ and $p$ is a prime number that does not divide the $\vr$-value at the merged vertex, Lorenzini \cite[Proposition~4.3]{L00} showed that the $p$-parts of $\mathcal{K}(G;\vd,\vr)$ and $\mathcal{K}(G_1;\vd_1,\vr_1) \oplus \mathcal{K}(G_2;\vd_2,\vr_2)$ are isomorphic. As a special case, this gives that
\[\mathcal{K}(G;\vd,\vr) \cong \mathcal{K}(G_1;\vd_1,\vr_1) \oplus \mathcal{K}(G_2;\vd_2,\vr_2)\]
for Laplacian arithmetical structures. The next theorem extends Lorenzini's result to all prime numbers~$p$ under the assumption that the greatest common divisor of the $\vr$-values at the vertices to be merged is~$1$. 

\begin{theorem}\label{thm:mergingcriticalgroup} 
Let $G_1$ and~$G_2$ be finite, connected graphs without loops, choose vertices $x \in V(G_1)$ and $y \in V(G_2)$, and let $G = G_1 \vee G_2$. Let $(\vd_1,\vr_1)$ and $(\vd_2,\vr_2)$ be arithmetical structures on $G_1$ and~$G_2$, respectively, and let $(\vd,\vr)$ be the arithmetical structure on~$G$ given by Proposition~\ref{prop:mergingconstruction}. If $\gcd(\vr_1(x),\vr_2(y)) = 1$, then
\[\mathcal{K}(G;\vd,\vr) \cong \mathcal{K}(G_1;\vd_1,\vr_1) \oplus \mathcal{K}(G_2;\vd_2,\vr_2).\]
\end{theorem}

\begin{proof}
For $i=1,2$, let $\mathcal{D}_0({G_i})$ be the set of divisors on~$G_i$ that have degree~$0$ with respect to~$\vr_i$. There is a natural map $\phi \colon \mathcal{D}_0({G_1}) \oplus \mathcal{D}_0({G_2}) \to \mathcal{K}(G;\vd,\vr)$ that takes a pair $(\delta_1,\delta_2)$ to the equivalence class of the divisor~$\delta$ on~$G$ defined by 
\begin{equation}\label{eq:delta}
\delta(v) = \begin{cases}
 \delta_1(v) & \text{if } v \in V(G_1)\setminus\{x\}, \\
 \delta_2(v) & \text{if }v \in V(G_2)\setminus\{y\}, \\
 \delta_1(x) + \delta_2(y) & \text{if }v \text{ is the vertex of } G \text{ resulting from merging $x$ and~$y$,}
\end{cases}
\end{equation}
which is an element of $\mathcal{K}(G;\vd,\vr)$ because $\deg_{\vr}(\delta)=0$. We will show that $\phi$ induces an isomorphism~$\Phi$ between $\mathcal{K}(G_1;\vd_1,\vr_1) \oplus \mathcal{K}(G_2;\vd_2,\vr_2)$ and $\mathcal{K}(G;\vd,\vr)$ as described via the diagram
\[\begin{tikzcd}
\mathcal{D}_0({G_1}) \oplus \mathcal{D}_0({G_2}) \arrow{r}{\phi} \arrow{d}{\pi} & \mathcal{K}(G;\vd,\vr),\\
\mathcal{K}(G_1;\vd_1,\vr_1) \oplus \mathcal{K}(G_2;\vd_2,\vr_2) \arrow[swap,dashed]{ur}{\Phi} &
\end{tikzcd}\]
where $\pi$ is the natural quotient map. To do this, it suffices to show that $\ker(\phi) = \ker(\pi)$ and that $\phi$ is surjective.

We first show that $\ker(\pi) \subseteq \ker(\phi)$. Note that $\ker(\pi)$ is generated by elements of the form $(\delta_v,0)$, where $\delta_v$ is the divisor on~$G_1$ obtained from the trivial divisor by firing once at $v \in V(G_1)$, and elements of the form $(0,\delta_v)$, where $\delta_v$ is the divisor on~$G_2$ obtained from the trivial divisor by firing once at $v \in V(G_2)$. By symmetry, it suffices to show that $(\delta_v,0) \in \ker(\phi)$ for all $v \in V(G_1)$. When $v \neq x$, then $\phi((\delta_v,0))$ is obtained from the trivial divisor on~$G$ by firing at~$v$, so $(\delta_v,0) \in \ker(\phi)$. When $v=x$, we will show how to obtain $\phi((\delta_x,0))$ from the trivial divisor on~$G$. Let $a=\vr_1(x)$ and $b=\vr_2(y)$. Since $\gcd(a,b)=1$, there are integers $s$ and~$t$ such that $sa+tb=1$. Let $\mathbf{w}\in\ZZ^{V(G)}$ be defined by
\[\mathbf{w}(v)=\begin{cases}
 -s\vr_1(v) & \text{if } v \in V(G_1)\setminus\{x\}, \\
 t\vr_2(v) & \text{if }v \in V(G_2)\setminus\{y\}, \\
 tb & \text{if }v \text{ is the vertex of } G \text{ resulting from merging $x$ and~$y$.}
\end{cases}\]
Since firing $\vr_1(v)$ times at each vertex~$v$ of~$G_1$ preserves divisors on~$G_1$, we have 
\[\sum_{v\in V(G_1)\setminus\{x\}}-\vr_1(v)\delta_v=\vr_1(x)\delta_x=a\delta_x.\]
Therefore, starting with the trivial divisor on~$G$ and firing $\mathbf{w}(v)$ times at each vertex $v \in V(G_1)\setminus\{x\}$ yields $\phi((sa\delta_x,0))$. Similarly, starting with the trivial divisor on~$G$ and firing $\mathbf{w}(v)$ times at each $v \in V(G_2)\setminus\{y\}$ yields $\phi((0,-tb\delta_y))$. As a result, the divisor obtained by starting with the trivial divisor on~$G$ and firing at each vertex $v \in V(G)$ exactly $\mathbf{w}(v)$ times is $\phi(((tb+sa)\delta_x, (tb-tb)\delta_y)) = \phi((\delta_x, 0))$, showing that $\ker(\pi) \subseteq \ker(\phi)$. 

We next show that $\ker(\phi) \subseteq \ker(\pi)$. Let $(\delta_1,\delta_2) \in \ker(\phi)$, and let $\delta$ be as in~\eqref{eq:delta}, so that $\phi((\delta_1,\delta_2)) = [\delta]$. Therefore, there is some $\mathbf{w} \in \ZZ^{V(G)}$ such that starting with the trivial divisor on~$G$ and firing $\mathbf{w}(v)$ times at each $v \in V(G)$ yields~$\delta$. Then, starting with the trivial divisor on~$G_1$ (respectively, $G_2$) and firing $\mathbf{w}(v)$ times at each $v \in V(G_1)$ (resp.\ each $v \in V(G_2)$) yields $\delta_1$ (resp.~$\delta_2$). (Here $\mathbf{w}(x)$ and $\mathbf{w}(y)$ are taken to be the value of $\mathbf{w}$ at the merged vertex of~$G$.) Hence $\pi((\delta_1,\delta_2)) = 0$, so $(\delta_1,\delta_2) \in \ker(\pi)$. 

Finally, we show that $\phi$ is surjective. Fix an element of $\mathcal{K}(G;\vd,\vr)$ represented by a divisor~$\delta$ on~$G$ with $\deg_{\vr}(\delta) = 0$. Since each vertex $v \in V(G_1)\setminus\{x\} \subseteq V(G)$ has $\vr(v)$ divisible by~$b$, the weighted sum of~$\delta$ on just the vertices in $V(G_1)\setminus\{x\}$ is $kb$ for some integer~$k$. Similarly, the weighted sum of~$\delta$ on just the vertices in $V(G_2)\setminus\{y\}$ is $\ell a$ for some integer~$\ell$. Since the value of~$\vr$ at the merged vertex is $ab$, if the value of~$\delta$ there is~$m$, we then have that $kb + \ell a + mab = \deg_{\vr}(\delta) = 0$. Working modulo~$a$ and using that $\gcd(a,b)=1$, we get that $a$ divides~$k$, so $k=ca$ for some integer~$c$. We then define divisors $\delta_1$ on~$G_1$ and $\delta_2$ on~$G_2$ by
\[\delta_1(v) = \begin{cases}
 -c & \text{for } v=x, \\
 \delta(v) & \text{for } v \in V(G_1)\setminus\{x\}
\end{cases} \qquad\text{and}\qquad
\delta_2(v) = \begin{cases}
 m+c & \text{for } v=y, \\
 \delta(v) & \text{for } v \in V(G_2)\setminus\{y\},
\end{cases}\]
so that $\deg_{\vr_1}(\delta_1) = -cab + kb = 0$, $\deg_{\vr_2}(\delta_2) = (m+c)ab + \ell a = 0$, and $\phi((\delta_1,\delta_2)) = [\delta]$. 
\end{proof}

\begin{remark}
Theorem~\ref{thm:mergingcriticalgroup} is not true without the assumption that the greatest common divisor of the $\vr$-values at the vertices to be merged is~$1$. Hower \cite[Proposition~7.9]{Hower} showed that the orders of the critical groups are related by 
\[\abs{\mathcal{K}(G;\vd,\vr)} = \abs{\mathcal{K}(G_1;\vd_1,\vr_1)} \cdot \abs{\mathcal{K}(G_2;\vd_2,\vr_2)} \cdot (\gcd(a,b))^2.\]
The explicit relationship between the critical groups when $\gcd(a,b) \neq 1$ is more subtle and beyond the scope of this paper.
\end{remark}

\section{Number of invariant factors of critical groups associated to trees}\label{sec:boundinginvariantfactors}

For the rest of this paper, we focus on trees. We use $\ell(T)$ to denote the number of leaves of a tree~$T$. In this section, we investigate how many invariant factors the critical group of an arithmetical structure on a tree can have, bounding this in terms of a decomposition of the tree into simpler trees. 

We first describe the type of decompositions we consider. A \emph{starlike tree} is a tree that has exactly one vertex of degree at least~$3$; we call this high-degree vertex the \emph{central vertex}. If a tree~$T$ has multiple vertices of degree at least~$3$ (that is, if it is not a starlike tree or a path graph), we define a \emph{starlike splitting} of~$T$ to be a $4$-tuple $(S,T',v_s,v_t)$, where $S$ is a starlike tree, $T'$ is a tree, $v_s$ is a leaf of~$S$ that is adjacent to its central vertex, $v_t$ is a vertex of~$T'$, and $T$ is obtained from $S$ and~$T'$ by merging the vertices $v_s$ and~$v_t$. 

\begin{example}\label{ex:onesplitting}
Let $T$ be the tree shown below. There is a starlike splitting of~$T$ into $S$ and~$T'$, as shown to the right below, obtained by splitting at vertex~$v$. On the other hand, no starlike splitting of~$T$ is obtained by splitting at vertex~$u$.
\end{example}

\begin{center}
\begin{tikzpicture}[scale=.9]
\node at (-2.05,-2) {$T$};
\draw (-2.5,-2.5) -- (-2.5,1.5) -- (3.6,1.5) -- (3.6,-2.5) -- (-2.5,-2.5);
\node[main_node, label=below:{$v$}] (1) at (0,0) {};
\node[main_node] (2) at (-1,0) {};
\node[main_node] (3) at (-1,-1) {};
\node[main_node] (4) at (-1,1) {};
\node[main_node] (5) at (-2,0) {};
\node[main_node, label=above:{$u$}] (6) at (1,0) {};
\node[main_node] (7) at (2,0) {};
\node[main_node] (8) at (3,0) {};
\node[main_node] (9) at (3,1) {};
\node[main_node] (10) at (3,-1) {};
\node[main_node] (11) at (1,-1) {};
\node[main_node] (12) at (1,-2) {};
\node[main_node] (13) at (0,1){};
\node[main_node] (15) at (-1,-2){};

\draw (9.center) -- (8.center) -- (7.center) -- (6.center) -- (1.center) -- (2.center) -- (3.center);
\draw (1.center) -- (13.center);
\draw (2.center) -- (4.center);
\draw (2.center) -- (5.center);
\draw (10.center) -- (8.center);
\draw (6.center) -- (11.center) -- (12.center);
\draw (15.center) -- (3.center);

\node at (4.25,-0.5) {$\longrightarrow$};
\end{tikzpicture}
\begin{tikzpicture}[scale=.9]
\node at (-3.05,-2) {$S$};
\node at (-.05,-2) {$T'$};
\draw (-.5,-2.5) -- (-.5,1.5) -- (3.6,1.5) -- (3.6,-2.5) -- (-.5,-2.5);
\draw (-.5,-2.5) -- (-.5,1.5) -- (-3.5,1.5) -- (-3.5,-2.5) -- (-.5,-2.5);

\node[main_node, label=below:{$v_t$}] (1) at (0,0) {};
\node[main_node] (2) at (-2,0) {};
\node[main_node] (3) at (-2,-1) {};
\node[main_node] (4) at (-2,1) {};
\node[main_node] (5) at (-3,0) {};
\node[main_node] (6) at (1,0) {};
\node[main_node] (7) at (2,0) {};
\node[main_node] (8) at (3,0) {};
\node[main_node] (9) at (3,1) {};
\node[main_node] (10) at (3,-1) {};
\node[main_node] (11) at (1,-1) {};
\node[main_node] (12) at (1,-2) {};
\node[main_node] (13) at (0,1){};
\node[main_node, label=below:{$v_s$}] (14) at (-1,0){};
\node[main_node] (15) at (-2,-2){};

\draw (9.center) -- (8.center) -- (7.center) -- (6.center) -- (1.center);
\draw (14.center) -- (2.center) -- (3.center);
\draw (1.center) -- (13.center);
\draw (2.center) -- (4.center);
\draw (2.center) -- (5.center);
\draw (10.center) -- (8.center);
\draw (6.center) -- (11.center) -- (12.center);
\draw (15.center) -- (3.center);

\end{tikzpicture}
\end{center}

A \emph{starlike decomposition} of a tree~$T$ is a collection $\{S_i\}_{i=1}^k$ of trees with at most one vertex of degree at least~$3$ that result from a sequence of starlike splittings. This means that, for all $i \in [k-1]$, the tree~$S_i$ is starlike with a leaf adjacent to its central vertex. The tree~$S_k$ is either a starlike tree or a path graph. Note that a given tree may have many different starlike decompositions. If $T$ is a starlike tree or a path graph, then its only starlike decomposition is $\{T\}$. 

\begin{example}\label{ex:twosplittings}
Two starlike decompositions of the tree~$T$ from Example~\ref{ex:onesplitting} are shown below. In the first, we split off $S_1$ at~$v$ and then split the remaining $T'$ at~$v'$ to get $S_2$ and~$S_3$. In the second, we first split at~$v'$ and then split at~$v$. These are two different starlike splittings of~$T$.
\end{example}

\begin{center}
\begin{tikzpicture}[scale=.85]
\node at (-2.05,-2) {$T$};
\draw (-2.5,-2.5) -- (-2.5,1.5) -- (3.6,1.5) -- (3.6,-2.5) -- (-2.5,-2.5);
\node[main_node, label=below:{$v$}] (1) at (0,0) {};
\node[main_node] (2) at (-1,0) {};
\node[main_node] (3) at (-1,-1) {};
\node[main_node] (4) at (-1,1) {};
\node[main_node] (5) at (-2,0) {};
\node[main_node] (6) at (1,0) {};
\node[main_node, label=above:{$v'$}] (7) at (2,0) {};
\node[main_node] (8) at (3,0) {};
\node[main_node] (9) at (3,1) {};
\node[main_node] (10) at (3,-1) {};
\node[main_node] (11) at (1,-1) {};
\node[main_node] (12) at (1,-2) {};
\node[main_node] (13) at (0,1){};
\node[main_node] (15) at (-1,-2){};

\draw (9.center) -- (8.center) -- (7.center) -- (6.center) -- (1.center) -- (2.center) -- (3.center);
\draw (1.center) -- (13.center);
\draw (2.center) -- (4.center);
\draw (2.center) -- (5.center);
\draw (10.center) -- (8.center);
\draw (6.center) -- (11.center) -- (12.center);
\draw (15.center) -- (3.center);

\node at (4.65,-0.5) {$\longrightarrow$};
\end{tikzpicture}
\begin{tikzpicture}[scale=.85]
\draw (-.5,-2.5) -- (-.5,1.5) -- (2.5,1.5) -- (2.5,-2.5) -- (-.5,-2.5);
\draw (-.5,-2.5) -- (-.5,1.5) -- (-3.5,1.5) -- (-3.5,-2.5) -- (-.5,-2.5);
\draw (2.5,1.5) -- (4.6,1.5) -- (4.6,-2.5) -- (2.5,-2.5);

\node at (-3.05,-2) {$S_1$};
\node at (-.05,-2) {$S_2$};
\node at (2.95,-2) {$S_3$};

\node[main_node, label=below:{$v_t$}] (1) at (0,0) {};
\node[main_node] (2) at (-2,0) {};
\node[main_node] (3) at (-2,-1) {};
\node[main_node] (4) at (-2,1) {};
\node[main_node] (5) at (-3,0) {};
\node[main_node] (6) at (1,0) {};
\node[main_node, label=above:{$v'_s$}] (7) at (2,0) {};
\node[main_node] (8) at (4,0) {};
\node[main_node] (9) at (4,1) {};
\node[main_node] (10) at (4,-1) {};
\node[main_node] (11) at (1,-1) {};
\node[main_node] (12) at (1,-2) {};
\node[main_node] (13) at (0,1){};
\node[main_node, label=below:{$v_s$}] (14) at (-1,0){};
\node[main_node] (15) at (-2,-2){};
\node[main_node, label=above:{$v'_t$}] (16) at (3,0){};

\draw (9.center) -- (8.center);
\draw (7.center) -- (6.center) -- (1.center);
\draw (14.center) -- (2.center) -- (3.center);
\draw (1.center) -- (13.center);
\draw (2.center) -- (4.center);
\draw (2.center) -- (5.center);
\draw (10.center) -- (8.center);
\draw (6.center) -- (11.center) -- (12.center);
\draw (15.center) -- (3.center);
\draw (16.center) -- (8.center);
\end{tikzpicture}
\end{center}

\begin{center}
\begin{tikzpicture}[scale=.85]
\node at (-2.05,-2) {$T$};
\draw (-2.5,-2.5) -- (-2.5,1.5) -- (3.6,1.5) -- (3.6,-2.5) -- (-2.5,-2.5);
\node[main_node, label=below:{$v$}] (1) at (0,0) {};
\node[main_node] (2) at (-1,0) {};
\node[main_node] (3) at (-1,-1) {};
\node[main_node] (4) at (-1,1) {};
\node[main_node] (5) at (-2,0) {};
\node[main_node] (6) at (1,0) {};
\node[main_node, label=above:{$v'$}] (7) at (2,0) {};
\node[main_node] (8) at (3,0) {};
\node[main_node] (9) at (3,1) {};
\node[main_node] (10) at (3,-1) {};
\node[main_node] (11) at (1,-1) {};
\node[main_node] (12) at (1,-2) {};
\node[main_node] (13) at (0,1){};
\node[main_node] (15) at (-1,-2){};

\draw (9.center) -- (8.center) -- (7.center) -- (6.center) -- (1.center) -- (2.center) -- (3.center);
\draw (1.center) -- (13.center);
\draw (2.center) -- (4.center);
\draw (2.center) -- (5.center);
\draw (10.center) -- (8.center);
\draw (6.center) -- (11.center) -- (12.center);
\draw (15.center) -- (3.center);

\node at (4.65,-0.5) {$\longrightarrow$};
\end{tikzpicture}
\begin{tikzpicture}[scale=.85]
\draw (-.5,-2.5) -- (-.5,1.5) -- (2.5,1.5) -- (2.5,-2.5) -- (-.5,-2.5);
\draw (-.5,-2.5) -- (-.5,1.5) -- (-3.5,1.5) -- (-3.5,-2.5) -- (-.5,-2.5);
\draw (2.5,1.5) -- (4.6,1.5) -- (4.6,-2.5) -- (2.5,-2.5);

\node[main_node, label=below:{$v_s$}] (1) at (0,0) {};
\node[main_node] (2) at (-2,0) {};
\node[main_node] (3) at (-2,-1) {};
\node[main_node] (4) at (-2,1) {};
\node[main_node] (5) at (-3,0) {};
\node[main_node] (6) at (1,0) {};
\node[main_node, label=above:{$v'_t$}] (7) at (2,0) {};
\node[main_node] (8) at (4,0) {};
\node[main_node] (9) at (4,1) {};
\node[main_node] (10) at (4,-1) {};
\node[main_node] (11) at (1,-1) {};
\node[main_node] (12) at (1,-2) {};
\node[main_node] (13) at (-1,1){};
\node[main_node, label=below:{$v_t$}] (14) at (-1,0){};
\node[main_node] (15) at (-2,-2){};
\node[main_node, label=above:{$v'_s$}] (16) at (3,0){};

\draw (9.center) -- (8.center);
\draw (7.center) -- (6.center) -- (1.center);
\draw (14.center) -- (2.center) -- (3.center);
\draw (14.center) -- (13.center);
\draw (2.center) -- (4.center);
\draw (2.center) -- (5.center);
\draw (10.center) -- (8.center);
\draw (6.center) -- (11.center) -- (12.center);
\draw (15.center) -- (3.center);
\draw (16.center) -- (8.center); 

\node at (-3.05,-2) {$S_3$};
\node at (-.05,-2) {$S_2$};
\node at (2.95,-2) {$S_1$};
\end{tikzpicture}
\end{center}

A starlike decomposition of a tree~$T$ is useful in finding representatives of divisors on~$T$ that are nonzero at relatively few vertices. 

\begin{lemma}\label{lem:divisors}
Let $T$ be a finite tree with at least~$2$ vertices, let $\vd \in \ZZ^{V(G)}$, let $\{S_i\}_{i=1}^k$ be a starlike decomposition of~$T$, and let $\ell_i=\ell(S_i)$. Every divisor on~$T$ is equivalent \textup{(}with respect to $\vd$\textup{)} to a divisor that is nonzero at at most $\sum_{i=1}^k(\ell_i-2) + 1$ vertices. 
\end{lemma}

\begin{proof}
We proceed by induction on~$k$, the number of trees in the starlike decomposition. If $k=1$, then $T=S_1$ is a starlike tree or path graph with $\ell_1$ leaves. Given a divisor on~$S_1$, we can borrow in along one tentacle until there are no chips at any vertex of the tentacle other than the base, as described in Remark~\ref{rem:borrowalongtentacle}. Then, we can borrow out along all the other tentacles so that there are only chips at $\ell_1-1$ leaf vertices. 

We now prove the inductive step. Suppose $T$ has a starlike decomposition into $k$~trees, where $k>1$. Consider the splitting of~$T$ into the starlike tree~$S_1$ and a tree~$T'$, where $T'$ has a starlike decomposition $\{S_i\}_{i=2}^k$ into $k-1$ trees.

First borrow in along one tentacle of~$S_1$ that does not include $v_s$ so there are no chips along this tentacle except at the base. Consider $T'$, with the $\vd$-values inherited from~$T$. By the inductive hypothesis, we can do some sequence of firings on~$T'$ to move all chips to $\sum_{i=2}^k(\ell_i-2) + 1$ of its vertices. Do the same firings on~$T$. Since the merged vertex~$v$ was the leaf of a length-$1$ tentacle of~$S_1$, we can then borrow out along all other tentacles of~$S_1$ so that the only vertices of $V(S_1)\setminus\{v\}$ with chips are $\ell_{1}-2$ leaves. The number of vertices of~$T$ that have chips is then at most $\ell_{1} - 2 + \sum_{i=2}^k (\ell_i-2) + 1 = \sum_{i=1}^{k} (\ell_i-2) + 1$. 
\end{proof}

\begin{remark}\label{rem:extend}
The next theorem uses an extension operation on arithmetical structures. Given an arithmetical structure $(\vd,\vr)$ on a graph~$G$ and a fixed vertex $v \in V(G)$, we can introduce a new vertex~$y$, connect it to~$v$, and set $\vr(y)=\vr(v)$. It is straightforward to check that this produces an arithmetical structure on the resulting graph and that the critical group of the new arithmetical structure is the same as that of the original arithmetical structure (see \cite[Section~1.8]{L89}).
\end{remark}

\begin{theorem} \label{thm:invariantfactorsboundtrees}
Let $T$ be a finite tree with at least~$2$ vertices, let $\{S_i\}_{i=1}^k$ be a starlike decomposition of~$T$, and let $\ell_i=\ell(S_i)$. 
\begin{enumerate}[label=\textup{(\alph*)},ref=\textup{\alph*}]
\item The number of invariant factors of the critical group of any arithmetical structure on~$T$ is at most $\sum_{i=1}^k(\ell_i-2)$.\label{part:invariantfactorbound}
\item For every integer~$t$ satisfying $0 \leq t \leq \sum_{i=1}^k(\ell_i-2)$, there is an arithmetical structure on~$T$ whose critical group has $t$~invariant factors.\label{part:invariantfactorconstruction}
\end{enumerate}
\end{theorem}

\begin{proof}
By Lemma~\ref{lem:divisors}, every divisor on~$T$ is equivalent under firing and borrowing to one that is nonzero at at most $\sum_{i=1}^k(\ell_i-2) + 1$ vertices. Using Theorem~\ref{thm:invariantfactorbound}, this implies that the critical group of any arithmetical structure on~$T$ has at most $\sum_{i=1}^k(\ell_i-2)$ nontrivial invariant factors, proving~(\ref{part:invariantfactorbound}).

Given an integer~$t$ satisfying $0 \leq t \leq \sum_{i=1}^k(\ell_i-2)$, we will construct an arithmetical structure on~$T$ whose critical group has exactly $t$~invariant factors. Let $\{t_i\}_{i=1}^k$ be integers satisfying $t = \sum_{i=1}^k$ and $0 \leq t_i \leq \ell_i-2$ for all $i \in [k]$. We first construct an arithmetical structure on a star graph with~$\ell_i$ leaves that has critical group $(\ZZ/2\ZZ)^{t_i}$. If $t_i=0$, take the Laplacian arithmetical structure, with each $\vr$-value equal to~$1$. If $t_i$ is even and positive, use the $\vr$-value~$2$ at the central vertex with~$1$ at $t_i+2$ of the leaves and $2$ at the remaining leaves. If $t_i$ is odd, use the $\vr$-value~$4$ at the central vertex with $1$ at $2$~of the leaves, $2$ at $t_i$~of the leaves, and $4$ at the remaining leaves. By \cite[Theorem~1]{A24}, each of these arithmetical structures has critical group $(\ZZ/2\ZZ)^{t_i}$. 

By extending at leaves, as described in Remark~\ref{rem:extend}, we can create tentacles of any desired lengths without changing the corresponding critical groups. Therefore, for all $i \in [k]$, we can obtain an arithmetical structure on~$S_i$ with critical group $(\ZZ/2\ZZ)^{t_i}$. Moreover, since each arithmetical structure we constructed on a star graph has $\vr$-value~$1$ at some leaf, we can arrange that for each $i<k$ the leaf of~$S_i$ that will be merged has $\vr$-value~$1$. 

Now we recompose $T$ from the~$S_i$. At each step, the $\vr$-value at the leaf of~$S_i$ to be merged is~$1$, so Theorem~\ref{thm:mergingcriticalgroup} gives that the construction of Proposition~\ref{prop:mergingconstruction} yields an arithmetical structure whose critical group is the direct sum of the previous critical groups. Thus we obtain an arithmetical structure on~$T$ with critical group $\bigoplus_{i=1}^k(\ZZ/2\ZZ)^{t_i}\cong(\ZZ/2\ZZ)^{t}$, proving~(\ref{part:invariantfactorconstruction}). 
\end{proof}

Using Theorem~\ref{thm:invariantfactorsboundtrees}, we can characterize which trees only admit arithmetical structures with cyclic critical groups and which only admit arithmetical structures with trivial critical groups.

\begin{corollary}\label{cor:alwayscyclic}
If $T$ is a finite tree, then the critical group of every arithmetical structure on~$T$ is cyclic if and only if 
\begin{enumerate}[label=\textup{(\alph*)},ref=\textup{\alph*}]
\item $T$ has does not have non-adjacent vertices of degree at least~$3$, and\label{part:noadjacentofdegree3}
\item $T$ has no vertices of degree at least~$4$.\label{part:nodegree4}
\end{enumerate}
Furthermore, the critical group of every arithmetical structure on~$T$ is trivial if and only if $T$ is a path graph.
\end{corollary}

\begin{proof}
By Theorem~\ref{thm:invariantfactorsboundtrees}(\ref{part:invariantfactorconstruction}), it suffices to show that $\sum_{i=1}^k(\ell_i-2) \leq 1$ if and only if both (\ref{part:noadjacentofdegree3}) and~(\ref{part:nodegree4}) hold. If (\ref{part:noadjacentofdegree3}) and~(\ref{part:nodegree4}) hold, then $T$ must be a path graph, be a starlike tree with three leaves, or have two adjacent vertices of degree~$3$ and no other vertices of degree at least~$3$. If $T$ is a path graph, then $\sum_{i=1}^k(\ell_i-2) = \ell_1-2 \leq 0$. If $T$ is a starlike tree with three leaves, then 
\[\sum_{i=1}^k(\ell_i-2) = 3-2 = 1.\] 
If $T$ has two adjacent vertices of degree~$3$ and no other vertices of degree at least~$3$, then it has a starlike decomposition into a starlike graph with three leaves and a path graph, so thus 
\[\sum_{i=1}^k(\ell_i-2) = (3-2)+(2-2) = 1.\] 

If (\ref{part:noadjacentofdegree3}) and~(\ref{part:nodegree4}) do not both hold, then $T$ has two non-adjacent vertices of degree at least~$3$ or has a vertex of degree at least~$4$. If $T$ has two non-adjacent vertices of degree at least~$3$, then in the initial splitting of~$T$ into $S_1$ and~$T'$ we have that $T'$ is not a path graph, so $\sum_{i=2}^k(\ell_i-2) \geq 1$. As $\ell_1 \geq 3$, this means
\[\sum_{i=1}^k(\ell_i-2) = (\ell_1-2) + \sum_{i=2}^k(\ell_i-2) \geq 1+1 = 2.\]
Now suppose $T$ has a vertex~$v$ of degree at least~$4$ and does not have non-adjacent vertices of degree at least~$3$. If $T$ is a starlike tree, then $k=1$, so $\sum_{i=1}^k(\ell_i-2) = \ell(T)-2 \geq 2$. If $T$ is not a starlike tree, then $v$ is adjacent to exactly one other vertex of degree at least~$3$. This means $T$ has a starlike decomposition $\{S_1, S_2\}$ with $\ell_1\geq4$ and $\ell_2\geq2$, so 
\[\sum_{i=1}^k(\ell_i-2) = (\ell_1-2) + (\ell_2-2) \geq 2+0 = 2.\]

To prove the last statement, by Theorem~\ref{thm:invariantfactorsboundtrees}(\ref{part:invariantfactorconstruction}) it suffices to show that $\sum_{i=1}^k(\ell_i-2) \leq 0$ if and only if $T$ is a path graph. If $T$ is a path graph, then $\sum_{i=1}^k(\ell_i-2) = \ell_1-2 \leq 0$. Conversely, if $\sum_{i=1}^k(\ell_i-2) \leq 0$, then $k=1$ and $\ell_1\leq2$; thus $T$ is a path graph.
\end{proof}

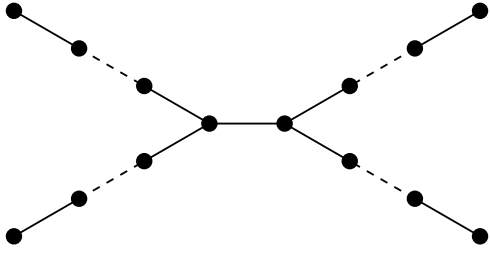
\begin{figure} 
\begin{center}
\begin{tikzpicture}[scale=1]
\node[main_node] (1) at (0,0) {};
\node[main_node] (2a) at (150:1) {};
\node[main_node] (2b) at (150:2) {};
\node[main_node] (2c) at (150:3) {};
\node[main_node] (3a) at (210:1) {};
\node[main_node] (3b) at (210:2) {};
\node[main_node] (3c) at (210:3) {};
\node[main_node] (4a) at ($(1,0)+(30:1)$) {};
\node[main_node] (4b) at ($(1,0)+(30:2)$) {};
\node[main_node] (4c) at ($(1,0)+(30:3)$) {};
\node[main_node] (5a) at ($(1,0)+(330:1)$) {};
\node[main_node] (5b) at ($(1,0)+(330:2)$) {};
\node[main_node] (5c) at ($(1,0)+(330:3)$) {};
\node[main_node] (6) at (1,0) {};
\draw (2a.center) -- (1.center) -- (3a.center);
\draw (4a.center) -- (6.center) -- (5a.center);
\draw (1.center) -- (6.center);
\draw (2c.center) -- (2b.center);
\draw (3c.center) -- (3b.center);
\draw (4c.center) -- (4b.center);
\draw (5c.center) -- (5b.center);
\draw[dashed] (2b.center) -- (2a.center);
\draw[dashed] (3b.center) -- (3a.center);
\draw[dashed] (4b.center) -- (4a.center);
\draw[dashed] (5b.center) -- (5a.center);
\end{tikzpicture}
\end{center}
\caption{Tree for which all arithmetical structures have cyclic critical group.}\label{fig:xwing}
\end{figure}

\begin{remark}
Corollary~\ref{cor:alwayscyclic} implies that the trees for which all arithmetical structures have cyclic critical group are exactly those shown in Figure~\ref{fig:xwing}, where the possibility that drawn tentacles are not present is allowed. 
\end{remark}

\section{Splitting irregularity numbers and \texorpdfstring{$2$}{2}-matching numbers}\label{sec:invariants}

In this section, we introduce an invariant of trees that we call the splitting irregularity number and show how it is related to the $2$-matching number. We then reinterpret Theorem~\ref{thm:invariantfactorsboundtrees} in terms of splitting irregularity numbers and $2$-matching numbers. 

\subsection{The splitting irregularity number} 
Suppose $(S,T',v_s,v_t)$ is a starlike splitting of a tree~$T$. If $v_t$ has degree~$1$ in~$T'$, we call this splitting \emph{regular}; if it has degree greater than~$1$, we call the splitting \emph{irregular}. In the first starlike decomposition in Example~\ref{ex:twosplittings}, the first splitting is irregular because $v_t$ has degree~$2$, while second splitting is regular because $v'_t$ has degree~$1$. We define the \emph{splitting irregularity number} of a tree~$T$ to be the number of irregular splittings in any starlike decomposition. In the next proposition, we show that this number does not depend on the sequence of starlike splittings that decompose~$T$ and is thus a well-defined invariant of~$T$.

\enlargethispage*{\baselineskip}

\begin{proposition}\label{prop:iotawelldefined}
Let $T$ be a finite tree, let $\{S_i\}_{i=1}^k$ be a starlike decomposition of~$T$ resulting from a sequence of starlike splittings of which $\iota$ are irregular, and let $\ell_i=\ell(S_i)$. We have that 
\begin{enumerate}[label=\textup{(\alph*)},ref=\textup{\alph*}]
\item $\ell(T)-2=\sum_{i=1}^k(\ell_i-2)+\iota$,\label{part:iotaandleaves} 
\item $\sum_{i=1}^k(\ell_i-2)$ is independent of the choice of starlike decomposition, and\label{part:leavessumindependence}
\item $\iota$ is independent of the sequence of starlike splittings.\label{part:iotawelldefined}
\end{enumerate} 
\end{proposition}

\begin{proof}
We first prove part~(\ref{part:iotaandleaves}). If $k=1$, then the statement is clearly true, so let $k \geq 2$. Suppose $(S,T',v_s,v_t)$ is a starlike splitting of~$T$. We have that 
\[\ell(T) - 2 = \ell(S) - 2 + \ell(T') - 2 + \begin{cases}
 1 & \text{if $(S,T',v_s,v_t)$ is irregular,} \\
 0 & \text{if $(S,T',v_s,v_t)$ is regular,}
\end{cases}\]
because $v_t$ is a leaf exactly when this splitting is regular. If $T$ has starlike decomposition $\{S_i\}_{i=1}^k$, then we can take $S=S_1$ so that $\ell(S)=\ell_1$ and take $T'$ to be a tree with induced starlike decomposition $\{S_i\}_{i=2}^k$. Since $\iota$ is the number of irregular splittings in the starlike decomposition, using the above equation $k-1$ times gives $\ell(T)-2=\sum_{i=1}^k(\ell_i-2)+\iota$.

Part~(\ref{part:leavessumindependence}) follows from Theorem~\ref{thm:invariantfactorsboundtrees}. More specifically, suppose a tree~$T$ has two starlike decompositions $\{S_{i}\}_{i=1}^{k}$ and $\{S'_{i}\}_{i=1}^{k'}$, with $\ell_i=\ell(S_i)$ and $\ell'_i=\ell(S_i')$. Theorem~\ref{thm:invariantfactorsboundtrees}(\ref{part:invariantfactorconstruction}) applied to $\{S_{i}\}_{i=1}^{k}$ gives that there is an arithmetical structure on~$T$ whose critical group has $\sum_{i=1}^k(\ell_i-2)$ invariant factors, while Theorem~\ref{thm:invariantfactorsboundtrees}(\ref{part:invariantfactorbound}) applied to $\{S'_{i}\}_{i=1}^{k'}$ gives that this critical group has at most $\sum_{i=1}^{k'}(\ell'_i-2)$ invariant factors. Thus $\sum_{i=1}^k(\ell_i-2) \leq \sum_{i=1}^{k'}(\ell'_i-2)$. Similarly, $\sum_{i=1}^{k'}(\ell'_i-2) \leq \sum_{i=1}^{k}(\ell_i-2)$, so therefore $\sum_{i=1}^k(\ell_i-2) = \sum_{i=1}^{k'}(\ell'_i-2)$.

Parts (\ref{part:iotaandleaves}) and~(\ref{part:leavessumindependence}) imply part~(\ref{part:iotawelldefined}). 
\end{proof}

Proposition~\ref{prop:iotawelldefined}(\ref{part:iotawelldefined}) shows that a tree~$T$ has a well-defined splitting irregularity number, which we henceforth denote by $\iota(T)$. For the tree~$T$ in Example~\ref{ex:twosplittings}, we have $\iota(T)=1$ since there is one irregular splitting in the given starlike decomposition. 

The next proposition determines the possible values of the splitting irregularity number in terms of the number of leaves of a tree. It also characterizes which trees have positive splitting irregularity number. 

\begin{proposition} 
The splitting irregularity number has the following properties.
\begin{enumerate}[label=\textup{(\alph*)},ref=\textup{\alph*}]
\item If $T$ is a tree with at least~$2$ vertices, then $0 \leq \iota(T) \leq \ell(T)/2-1$.\label{part:iotabounds}
\item Given integers $a$ and~$b$ with $0 \leq a \leq b/2-1$, there is a tree~$T$ with $\iota(T)=a$ and $\ell(T)=b$.\label{part:iotaconstruction}
\item If $T$ is a tree, then $\iota(T)$ is positive if and only if $T$ has adjacent vertices of degree at least~$3$.\label{part:iotapositivity}
\end{enumerate}
\end{proposition}

\begin{proof}
Let $\{S_i\}_{i=1}^k$ be a starlike decomposition of~$T$. To show~(\ref{part:iotabounds}), we proceed by induction on~$k$. If $k=1$, then $\iota(T)=0$ and $\ell(T) \geq 2$, so the result holds. Now suppose $k \geq 2$. If the splitting of~$T$ into $S_1$ and~$T'$ is regular, then using the inductive hypothesis we have 
\[0 \leq \iota(T) = \iota(T') \leq \ell(T')/2-1 < \ell(T)/2-1.\]
If the splitting of~$T$ into $S$ and~$T'$ is irregular, then using the inductive hypothesis and $\ell(T) \geq \ell(T')+2$ we have 
\[0 \leq \iota(T) = \iota(T')+1 \leq \ell(T')/2 \leq \ell(T)/2-1.\]

To show~(\ref{part:iotaconstruction}), let $a$ and~$b$ be integers with $0 \leq a \leq b/2-1$, which implies $2a+2 \leq b$. Let $P_3$ denote the path graph with three vertices and $\Star_3$ the star graph with three leaves. To obtain $T$, take the wedge sum of $P_3$ at its center vertex with $\Star_3$ at one of its leaves, then take the wedge sum of the resulting tree at one of its non-leaf vertices with another $\Star_3$ at one of its leaves, repeating this process until $a$~copies of $\Star_3$ have been added. Then take the wedge sum of the resulting tree at one of its leaves with $\Star_3$ at one of its leaves, repeating this process until $b-2a-2$ more copies of $\Star_3$ have been added. Since this construction can be undone via a starlike decomposition with exactly $a$~irregular splittings, we have that $\iota(T)=a$. The number of leaves in the tree resulting from this construction is $\ell(T) = 2+2a+b-2a-2 = b$.

To show~(\ref{part:iotapositivity}), note that if $\iota(T)$ is positive then it is clear that $T$ has adjacent vertices of degree at least~$3$. In the other direction, suppose $T$ has a pair of adjacent vertices of degree at least~$3$. When doing a starlike decomposition, if all previous splittings were regular, the splitting that separates these vertices would be irregular, making $\iota(T)$ positive. 
\end{proof}

The value of $\iota(T)$ is generally not obvious without finding a starlike decomposition of~$T$. Consider the following example. 
\begin{center}
\begin{tikzpicture}[scale=.9]
\node at (-2.1,-1) {$T_1$};
\draw (-2.6,-1.5) -- (-2.6,1.5) -- (2.6,1.5) -- (2.6,-1.5) -- (-2.6,-1.5);
\node[main_node] (1) at (0,0) {};
\node[main_node] (2) at (-1, -0) {};
\node[main_node] (3) at (-1, -1) {};
\node[main_node] (4) at (-1,1) {};
\node[main_node] (5) at (-2, 0) {};
\node[main_node] (6) at (1,0) {};
\node[main_node] (7) at (2,0) {};
\node[main_node] (11) at (1,-1) {};
\node[main_node] (13) at (0,1){};

\draw (7.center) -- (6.center) -- (1.center) -- (2.center) -- (3.center);
\draw (2.center) -- (4.center);
\draw (2.center) -- (5.center);
\draw (6.center) -- (11.center);
\draw (1.center) -- (13.center);
\end{tikzpicture} \quad vs. \quad
\begin{tikzpicture}[scale=.9]
\node at (-2.1,-1) {$T_2$};
\draw (-2.6,-1.5) -- (-2.6,1.5) -- (2.6,1.5) -- (2.6,-1.5) -- (-2.6,-1.5);
\node[main_node] (1) at (0,0) {};
\node[main_node] (2) at (-1,0) {};
\node[main_node] (3) at (0,-1) {};
\node[main_node] (4) at (-1,1) {};
\node[main_node] (5) at (-2,0) {};
\node[main_node] (6) at (1,0) {};
\node[main_node] (7) at (2,0) {};
\node[main_node] (11) at (1,-1) {};
\node[main_node] (13) at (0,1){};

\draw (7.center) -- (6.center) -- (1.center) -- (2.center);
\draw (2.center) -- (4.center);
\draw (2.center) -- (5.center);
\draw (6.center) -- (11.center);
\draw (3.center) -- (1.center) -- (13.center);
\end{tikzpicture}
\end{center}
Here $T_1$ and~$T_2$ are trees with the same degree sequence: one vertex of degree~$4$, two vertices of degree~$3$, and six leaves. However, $\iota(T_1)=1$ while $\iota(T_2)=2$. Therefore, $\iota(T)$ depends on more than just the degree sequence of~$T$; one must consider the adjacencies between high-degree vertices and what happens to these vertices under starlike splittings.

\subsection{\texorpdfstring{$2$}{2}-matching numbers}
The splitting irregularity number of a tree can also be computed in terms of its $2$-matching number. In the context of critical ideals, $2$-matching numbers were previously considered in~\cite{CV15}. A \emph{$2$-matching} of a graph~$G$ is a subset of the edge set $E(G)$ with the property that, for each vertex~$v \in V(G)$, at most~$2$ edges in this subset are incident to~$v$. The \emph{$2$-matching number} $\nu_2(G)$ is the maximum cardinality of a $2$-matching of~$G$. In the special case of a starlike tree~$S$, a maximum $2$-matching is given by choosing all edges of~$S$ except for any $\ell(S)-2$ edges incident to the central vertex, meaning $\nu_2(S) = \abs{E(S)}-\ell(S)+2$. As another example, if $T$ is the tree in Figure~\ref{fig:2matchingexample}, then $\nu_2(T)=14$. 

For trees, the $2$-matching number is additive under starlike splittings.

\begin{figure}[t]
\centering
\begin{tikzpicture}[scale=.45]
\node[main_node] (1) at (0,0) {};
\node[main_node] (2) at (30:2) {};
\node[main_node] (3) at (150:2) {};
\node[main_node] (4) at (270:2) {};
\node[main_node] (5) at ($(2)+(-15:2)$) {};
\node[main_node] (6) at ($(2)+(75:2)$) {};
\node[main_node] (7) at ($(3)+(195:2)$) {};
\node[main_node] (8) at ($(3)+(105:2)$) {};
\node[main_node] (9) at ($(4)+(225:2)$) {};
\node[main_node] (10) at ($(4)+(-45:2)$) {}; 
\node[main_node] (11) at ($(5)+(30:2)$) {};
\node[main_node] (12) at ($(5)+(-60:2)$) {};
\node[main_node] (13) at ($(6)+(30:2)$) {};
\node[main_node] (14) at ($(6)+(120:2)$) {};
\node[main_node] (15) at ($(8)+(60:2)$) {};
\node[main_node] (16) at ($(8)+(150:2)$) {};
\node[main_node] (17) at ($(7)+(150:2)$) {};
\node[main_node] (18) at ($(7)+(240:2)$) {}; 
\node[main_node] (19) at ($(10)+(0:2)$) {};
\node[main_node] (20) at ($(10)+(270:2)$) {}; 
\node[main_node] (21) at ($(9)+(180:2)$) {};
\node[main_node] (22) at ($(9)+(270:2)$) {};
 
\draw[highlight] (1.center) -- (2.center);
\draw[highlight] (1.center) -- (3.center);
\draw (1.center) -- (4.center);
\draw (2.center) -- (5.center);
\draw (2.center) -- (6.center);
\draw (3.center) -- (7.center);
\draw (3.center) -- (8.center);
\draw (4.center) -- (9.center);
\draw (4.center) -- (10.center);
\draw[highlight] (5.center) -- (11.center);
\draw[highlight] (5.center) -- (12.center);
\draw[highlight] (6.center) -- (13.center);
\draw[highlight] (6.center) -- (14.center);
\draw[highlight] (8.center) -- (15.center);
\draw[highlight] (8.center) -- (16.center);
\draw[highlight] (7.center) -- (17.center);
\draw[highlight] (7.center) -- (18.center);
\draw[highlight] (10.center) -- (19.center);
\draw[highlight] (10.center) -- (20.center);
\draw[highlight] (9.center) -- (21.center);
\draw[highlight] (9.center) -- (22.center);

\node[main_node] at (1) {};
\node[main_node] at (2) {};
\node[main_node] at (3) {};
\node[main_node] at (4) {};
\node[main_node] at (5) {};
\node[main_node] at (6) {};
\node[main_node] at (7) {};
\node[main_node] at (8) {};
\node[main_node] at (9) {};
\node[main_node] at (10) {}; 
\node[main_node] at (11) {};
\node[main_node] at (12) {};
\node[main_node] at (13) {};
\node[main_node] at (14) {};
\node[main_node] at (15) {};
\node[main_node] at (16) {};
\node[main_node] at (17) {};
\node[main_node] at (18) {}; 
\node[main_node] at (19) {};
\node[main_node] at (20) {}; 
\node[main_node] at (21) {};
\node[main_node] at (22) {};
\end{tikzpicture}
\caption{The thick, red edges form a maximal $2$-matching of this tree.}\label{fig:2matchingexample}
\end{figure}

\begin{lemma}\label{lem:2matchingsplitting}
If $T$ is a finite tree and $(S,T',v_s,v_t)$ is a starlike splitting of~$T$, then $\nu_2(T) = \nu_2(S) + \nu_2(T')$.
\end{lemma}

\begin{proof}
Given a $2$-matching of~$T$, restricting induces $2$-matchings of $S$ and~$T'$, so we immediately have that $\nu_2(T) \leq \nu_2(S) + \nu_2(T')$. A maximal $2$-matching on~$S$ is given by using all edges except the edge incident to the leaf where the merging takes place and $\ell(S)-3$ other edges incident to the central vertex. Taking the union of this with any maximal $2$-matching on~$T'$ gives a $2$-matching on~$T$, proving that $\nu_2(T) \geq \nu_2(S) + \nu_2(T')$ and hence that $\nu_2(T) = \nu_2(S) + \nu_2(T')$.
\end{proof}

For a tree, there is a relationship between its splitting irregularity number and $2$-matching number.

\begin{proposition}\label{prop:iotaand2matchingnumber}
If $T$ is a finite tree with at least~$2$ vertices, then 
\[\ell(T)-2-\iota(T) = \abs{E(T)}-\nu_2(T).\] 
\end{proposition}

\begin{proof}
Let $\{S_i\}_{i=1}^k$ be a starlike decomposition of~$T$, and let $\ell_i=\ell(S_i)$. Since each~$S_i$ is a starlike tree or a path graph, we have that $\nu_2(S_i) = \abs{E(S_i)}-\ell_i+2$ and hence $\ell_i-2 = \abs{E(S_i)}-\nu_2(S_i)$. Summing over all~$i$ and using Lemma~\ref{lem:2matchingsplitting}, we obtain
\[\sum_{i=1}^k(\ell_i-2) = \sum_{i=1}^k(\abs{E(S_i)}-\nu_2(S_i)) = \abs{E(T)}-\nu_2(T).\] 
Combining this with Proposition~\ref{prop:iotawelldefined}, we get $\ell(T)-2-\iota(T) = \abs{E(T)}-\nu_2(T)$. 
\end{proof}

As a consequence of Propositions \ref{prop:iotawelldefined} and~\ref{prop:iotaand2matchingnumber}, we obtain the following reinterpretation of Theorem~\ref{thm:invariantfactorsboundtrees}, which recovers and extends \cite[Theorem~3.9]{CV15}.

\begin{theorem}\label{thm:invariantfactorsboundtrees2}
Let $T$ be a finite tree with at least~$2$ vertices. 
\begin{enumerate}[label=\textup{(\alph*)},ref=\textup{\alph*}]
\item The number of invariant factors of the critical group of any arithmetical structure on~$T$ is at most $\ell(T)-2-\iota(T)$, or equivalently at most $\abs{E(T)}-\nu_2(T)$.\label{part:invariantfactorbound2}
\item For any integer~$t$ satisfying $0 \leq t \leq \ell(T)-\iota(T)-2$ there is an arithmetical structure on~$T$ whose critical group has $t$~invariant factors.\label{part:invariantfactorconstruction2}
\end{enumerate}
\end{theorem}

\begin{proof}
The bound of $\ell(T)-2-\iota(T)$ in part~(\ref{part:invariantfactorbound2}) follows immediately from Theorem~\ref{thm:invariantfactorsboundtrees}(\ref{part:invariantfactorbound}) and Proposition~\ref{prop:iotawelldefined}(\ref{part:iotaandleaves}). Proposition~\ref{prop:iotaand2matchingnumber} gives that this is equivalent to $\abs{E(T)}-\nu_2(T)$. Part~(\ref{part:invariantfactorconstruction2}) follows from Theorem~\ref{thm:invariantfactorsboundtrees}(\ref{part:invariantfactorconstruction}). 
\end{proof}

We conclude this section by using Proposition~\ref{prop:iotaand2matchingnumber} to establish a result about how the splitting irregularity number changes under subdivision of an edge. This will be used later, in the proof of Theorem~\ref{thm:constructgroups}.

\begin{lemma}\label{lem:2matchingsubdivision}
If $T$ is a finite tree and $\widetilde{T}$ is obtained from~$T$ by subdividing a single edge, then either
\begin{enumerate}[label=\textup{(\alph*)},ref=\textup{\alph*}]
\item $\nu_2\bigl(\widetilde{T}\bigr) = \nu_2(T)$ and $\iota\bigl(\widetilde{T}\bigr) = \iota(T)-1$, or 
\item $\nu_2\bigl(\widetilde{T}\bigr) = \nu_2(T)+1$ and $\iota\bigl(\widetilde{T}\bigr) = \iota(T)$. 
\end{enumerate}
\end{lemma}

\begin{proof}
Suppose $\widetilde{T}$ is obtained by subdividing an edge $e \in E(T)$ to get edges $e_1, e_2 \in E\bigl(\widetilde{T}\bigr)$. Given a $2$-matching on~$T$, define a $2$-matching on~$\widetilde{T}$ by including both $e_1$ and~$e_2$ if $e$ is included in the $2$-matching on~$T$, including neither $e_1$ nor~$e_2$ if $e$ is not included in the $2$-matching on~$T$, and including other edges if and only if they are in the $2$-matching on~$T$. Thus, to each $2$-matching on~$T$, there is a $2$-matching on~$\widetilde{T}$ that includes at least as many edges, so $\nu_2\bigl(\widetilde{T}\bigr) \geq \nu_2(T)$. Given a $2$-matching on~$\widetilde{T}$, define a $2$-matching on~$T$ by including $e$ if and only if $e_1$ and~$e_2$ are both included in the $2$-matching on~$\widetilde{T}$ and including other edges if and only if they are in the $2$-matching on~$\widetilde{T}$. Thus, to each $2$-matching on~$\widetilde{T}$ there is a $2$-matching on~$T$ that has no less than one fewer edge, so $\nu_2(T) \geq \nu_2\bigl(\widetilde{T}\bigr)-1$. It follows that $\nu_2\bigl(\widetilde{T}\bigr) = \nu_2(T)$ or $\nu_2\bigl(\widetilde{T}\bigr) = \nu_2(T)+1$. 

Proposition~\ref{prop:iotaand2matchingnumber} gives that $\iota(T) = \ell(T)-2-\abs{E(T)}+\nu_2(T)$. We know that $\ell\bigl(\widetilde{T}\bigr) = \ell(T)$ and $\abs[\big]{E\bigl(\widetilde{T}\bigr)} = \abs{E(T)}+1$. If $\nu_2\bigl(\widetilde{T}\bigr) = \nu_2(T)$, then $\iota\bigl(\widetilde{T}\bigr) = \iota(T)-1$. If $\nu_2\bigl(\widetilde{T}\bigr) = \nu_2(T)+1$, then $\iota\bigl(\widetilde{T}\bigr) = \iota(T)$. 
\end{proof}

\section{Arithmetical structures on trees with prescribed critical groups}\label{sec:constructgroups}

In the previous section, we showed that the critical group of each arithmetical structure on a tree~$T$ has at most $\ell(T)-2-\iota(T)$ invariant factors. In this section, we prove a partial converse: Given a tree~$T$ and a finite abelian group $\mathcal{G}$ that has at most $\ell(T)-2-\beta$ invariant factors with $0 \leq \beta \leq \iota(T)$, there is a subdivision $\widetilde{T}$ of~$T$ with $\iota\bigl(\widetilde{T}\bigr) = \beta$ and an arithmetical structure on~$T'$ that has $\mathcal{G}$ as critical group.

In order to provide a way of constructing arithmetical structures on trees with prescribed critical groups, we must understand critical groups of arithmetical structures on starlike trees. The next theorem follows immediately from \cite[Theorem~2.1]{L92} and shows how to compute the critical group of an arithmetical structure on a starlike graph. Before becoming aware of~\cite{L92}, we previously proved a special case of this result in \cite[Theorem~1]{A24}; the techniques used there can also be extended to prove this result.

\begin{theorem}\label{thm:criticalgroupstarlike}
Let $S$ be a starlike tree with central vertex~$v_0$ and $\ell$~leaves $v_1, v_2, \dots, v_\ell$. If $(\vd,\vr)$ is an arithmetical structure on~$S$, then
\[\mathcal{K}(S;\vd,\vr) \oplus (\ZZ/r_0\ZZ)^2 \cong \bigoplus_{i=1}^\ell \ZZ/d^*_i\ZZ,\]
where $r_i=\vr(v_i)$ and $d^*_i = r_0/r_i$. 
\end{theorem}

\begin{example}
Consider the starlike tree~$T$ and arithmetical structure $(\vd,\vr)$ given in Figure~\ref{fig:broomgraph}. This has $r_0=324$ and 
\[(d_i^*)_{i=1}^4 = (324/108, 324/18, 324/1,324/1) = (3,18,324,324).\] 
Theorem~\ref{thm:criticalgroupstarlike} gives 
\[\mathcal{K}(T;\vd,\vr) \oplus (\ZZ/324\ZZ)^2 \cong \ZZ/3\ZZ \oplus \ZZ/18\ZZ \oplus \ZZ/324\ZZ \oplus \ZZ/324\ZZ,\] 
so $\mathcal{K}(T;\vd,\vr) \cong \ZZ/3\ZZ \oplus \ZZ/18\ZZ$.
\end{example}

We next show how to construct an arithmetical structure with a prescribed critical group on a broom graph. A \emph{broom graph} with $j$~prongs is a starlike tree with $j+1$ tentacles of which at most one has length greater than~$1$.

\begin{proposition}\label{prop:broomgraphs}
Let $m$ be a positive integer, and let $\mathcal{G}$ be a finite abelian group with at most~$m$ invariant factors. There is an arithmetical structure on a broom graph with $m+1$ prongs whose critical group is~$\mathcal{G}$.
\end{proposition}

\begin{proof}
If $\mathcal{G}$ is the trivial group, we can take the Laplacian arithmetical structure on any broom graph with $m+1$ prongs, so assume $\mathcal{G}$ is nontrivial. Using the invariant factor decomposition, we have $\mathcal{G} \cong \ZZ/\alpha_1\ZZ \oplus \ZZ/\alpha_2\ZZ \oplus \cdots \oplus \ZZ/\alpha_m \ZZ$ for some positive integers~$\alpha_i$ with $\alpha_i\mid\alpha_{i+1}$ for all $i \in [m-1]$ and $\alpha_m \neq 1$. (If $\mathcal{G}$ has fewer than $m$~invariant factors, take some initial $\alpha_i$'s to be~$1$.) Let $T$ be a broom graph with $m+1$ prongs with central vertex~$v_0$, prongs $v_1, v_2, \dots, v_{m+1}$, and vertices $v_{m+2}, v_{m+3}, \dots, v_{n}$ along a tentacle, for some $n \geq m+2$ to be determined. Define $\vr(v_0) = \alpha_m^2$, $\vr(v_i) = \alpha_m^2/\alpha_i$ for each $i \in [m]$, and $\vr(v_{m+1})=1$. It is clear that each of $\vr(v_1), \vr(v_2), \dots, \vr(v_m)$ divides $\vr(v_0)$. Let $\vr(v_{m+2}) = -\sum_{i=1}^{m+1}\vr(v_i) \bmod{\vr(v_0)}$, which is nonzero because $\vr(v_i)$ is divisible by~$\alpha_m$ for all $i \in [m]$ but $\vr(v_{m+1})$ is not divisible by~$\alpha_m$. We can extend a tentacle in this direction by letting $\vr(v_{m+3}) = -\vr(v_{0}) \bmod{\vr(v_{m+2})}$ and successively letting $\vr(v_{j}) = -\vr(v_{j-2}) \bmod{\vr(v_{j-1})}$ for each $j \geq m+4$ as long as $\vr(v_{j-1}) \geq 2$. Since the $\vr$-values along the tentacle necessarily decrease, this process terminates and determines the length of the tentacle. Note that $\vr(v_{m+2})$ is one less than a multiple of~$\alpha_m$ and hence coprime to $\vr(v_0)$. This implies that $\vr(v_n)=1$; indeed, since the $\vr$-value at each vertex along this tentacle must be divisible by $\vr(v_n)$, if $\vr(v_n)$ were not~$1$ then $\vr(v_0)$ and $\vr(v_{m+2})$ would not be relatively prime. Theorem~\ref{thm:criticalgroupstarlike} gives that the critical group of this arithmetical structure is~$\mathcal{G}$.
\end{proof}

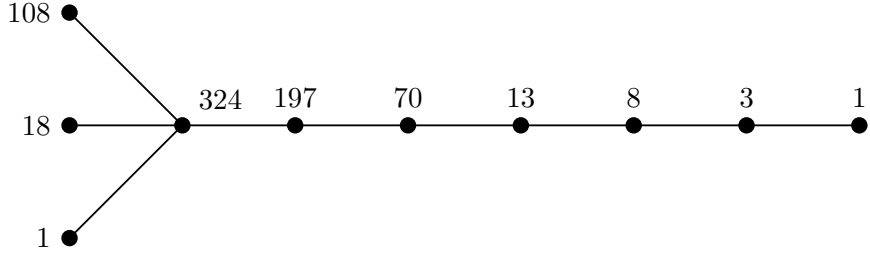
\begin{figure}
\centering
\begin{tikzpicture}[scale=1.5]
\node[main_node, label=above right:{$324$}] (c) at (0,0) {};
\node[main_node, label=left:{$108$}] (1) at (-1,1) {};
\node[main_node, label=left:{$18$}] (2) at (-1,0) {};
\node[main_node, label=left:{$1$}] (3) at (-1,-1) {};
\node[main_node, label=above:{$197$}] (4) at (1,0) {};
\node[main_node, label=above:{$70$}] (5) at (2,0) {};
\node[main_node, label=above:{$13$}] (6) at (3,0) {};
\node[main_node, label=above:{$8$}] (7) at (4,0) {};
\node[main_node, label=above:{$3$}] (8) at (5,0) {};
\node[main_node, label=above:{$1$}] (9) at (6,0) {};

\draw (1.center) -- (c.center) -- (2.center);
\draw (3.center) -- (c.center) -- (4.center);
\draw (4.center) -- (5.center) -- (6.center) -- (7.center) -- (8.center) -- (9.center);
\end{tikzpicture}
\caption{A broom graph with three prongs and the $\vr$-values of an arithmetical structure. The associated critical group is $\ZZ/3\ZZ \oplus \ZZ/18\ZZ$.}
\label{fig:broomgraph}
\end{figure}

As an example, to find an arithmetical structure on a broom graph with three prongs with critical group isomorphic to $\ZZ/3\ZZ \oplus \ZZ/18\ZZ$, the construction from this proof gives the graph and arithmetical structure in Figure~\ref{fig:broomgraph}. Indeed, we have $m=2$ with $\alpha_1=3$ and $\alpha_2=18$. Thus $\vr(v_0) = 18^2 = 324$, $\vr(v_1) = 18^2/3 = 108$, $\vr(v_2) = 18^2/18 = 18$, and $\vr(v_3)=1$, so $\vr(v_4) = -(108+18+1) \bmod{364} = 197$. The remaining tentacle is extended until it terminates at a vertex with an $\vr$-value of~$1$, resulting in the arithmetical structure shown. This type of extension along tentacles has previously been used in several papers, including \cite{L89,A20}.

Proposition~\ref{prop:broomgraphs} implies that every finite abelian group arises as the critical group of an arithmetical structure on a tree, and more specifically a broom graph. As another way to find an arithmetical structure on a tree with prescribed critical group, one could use Proposition~\ref{prop:broomgraphs} to build arithmetical structures on broom graphs with two prongs (i.e., bidents) whose critical groups are the invariant factors of the desired group. Since each of these arithmetical structures have $\vr$-values of~$1$ at two leaves, we can successively merge leaves with $\vr$-values of~$1$ to obtain an arithmetical structure on a tree that has no vertices of degree greater than~$3$. By Theorem~\ref{thm:mergingcriticalgroup}, this arithmetical structure has the desired critical group. While directly using Proposition~\ref{prop:broomgraphs} results in a tree with only one vertex of degree greater than~$2$, this second approach gives a tree with no vertex of degree greater than~$3$.

We can also try to produce arithmetical structures with prescribed critical groups on graphs with more specific properties, such as a given number of leaves and a given number of vertices of degree~$i$ for $i \geq 3$. The following theorem gives conditions under which this can be done, but it does not control the number of vertices of degree~$2$.

\begin{theorem}\label{thm:constructgroups}
Let $T$ be a finite tree, and let $\beta$ be an integer satisfying $0 \leq \beta \leq \iota(T)$. If $\mathcal{G}$ is a finite abelian group with $\ell(T)-2-\beta$ invariant factors, then there is a subdivision~$\widetilde{T}$ of~$T$ with $\iota\bigl(\widetilde{T}\bigr) = \beta$ and an arithmetical structure on~$\widetilde{T}$ whose critical group is~$\mathcal{G}$.
\end{theorem}

\begin{proof}
If $T$ is a path graph, then $\mathcal{G}$ can only be the trivial group, and any arithmetical structure on~$T$ has trivial critical group. Therefore we assume that $T$ is not a path graph. First suppose $\beta=\iota(T)$, and let $\{S_i\}_{i=1}^k$ be a starlike decomposition of~$T$. We proceed by induction on~$k$.

In the base case $k=1$, the tree~$T$ is starlike and $\iota(T)=0$, so $\beta=0$. Let $a_1, a_2, \dots, a_{\ell(T)}$ be the lengths of the tentacles of~$T$, and let $\mathcal{G}$ be a finite abelian group with at most $\ell(T)-2$ invariant factors. By Proposition~\ref{prop:broomgraphs}, there is an arithmetical structure on a broom graph with $\ell(T)-1$ prongs (and thus $\ell(T)$ leaves) with critical group~$\mathcal{G}$. Let $b$ be the length of the longest tentacle of this broom graph. By extending the tentacles of this broom graph and using the same $\vr$-values at the new vertices as described in Remark~\ref{rem:extend}, we obtain an arithmetical structure on a starlike tree~$\widetilde{T}$ that has tentacles of length $a_1, a_2, \dots, a_{\ell(T)-1}, \max\{a_{\ell(T)},b\}$ and where the critical group is still~$\mathcal{G}$. This~$\widetilde{T}$ can also be obtained from~$T$ by subdividing the $\ell(T)$-th tentacle of~$T$ (when $b>a_{\ell(T)}$) and satisfies $\iota\bigl(\widetilde{T}\bigr) = 0$, so the base case holds.

For the inductive step, suppose $k \geq 2$. Do a starlike splitting of~$T$ into $S_1$ with $\ell_1=\ell(S_1)$ and a tree~$T'$ with starlike decomposition $\{S_i\}_{i=2}^k$. Let $a_1, a_2, \dots, a_{\ell_1}$ be the lengths of the tentacles of~$S_1$, with $a_1=1$. If the splitting is regular, then $\iota(T)=\iota(T')$ and $\ell(T) = \ell_1+\ell(T')-2$; if the splitting is irregular, then $\iota(T)=\iota(T')+1$ and $\ell(T) = \ell_1+\ell(T')-1$. In either case, $\ell(T)-\ell_1-\iota(T) = \ell(T')-2-\iota(T')$. Let $\mathcal{G}$ be a finite abelian group with at most $\ell(T)-2-\iota(T)$ invariant factors. Write $\mathcal{G} \cong \mathcal{G}_1 \oplus \mathcal{G}_2$, where $\mathcal{G}_1$ has at most $\ell_1-2$ invariant factors and $\mathcal{G}_2$ has at most 
\[\ell(T)-2-\iota(T)-(\ell_1-2) = \ell(T)-\ell_1-\iota(T) = \ell(T')-2-\iota(T')\] 
invariant factors. By Proposition~\ref{prop:broomgraphs}, there is an arithmetical structure on a broom graph with $\ell_1-1$ prongs whose critical group is~$\mathcal{G}_1$. Let $b$ be the length of the longest tentacle of this broom graph. Moreover, this arithmetical structure has $\vr$-value~$1$ at one prong. Fixing this prong, extending the other tentacles of the broom graph, and using the same $\vr$-values at the new vertices as described in Remark~\ref{rem:extend}, we obtain an arithmetical structure on a graph~$\widetilde{S}_1$ that has tentacles of length $1, a_2, a_3, \dots, a_{\ell_{1}-1}, \max\{a_{\ell_1},b\}$, where the $\vr$-value at the leaf of the first tentacle is~$1$ and the critical group is~$\mathcal{G}_1$. This $\widetilde{S}_1$ can also be obtained by subdividing the $\ell(T)$-th tentacle of~$S_1$ (when $b>a_{\ell(T)}$). 

Since $T'$ has a starlike decomposition $\{S_i\}_{i=2}^k$ into $k-1$ trees, the inductive hypothesis gives that there is a subdivision~$\widetilde{T}'$ of~$T'$ on which there is an arithmetical structure with critical group~$\mathcal{G}_2$ and where $\iota\bigl(\widetilde{T}'\bigr)=\iota(T')$. By merging the leaf of the first tentacle of~$\widetilde{S}_1$ (where the $\vr$-value is~$1$) with the vertex of~$\widetilde{T}'$ that corresponds to the vertex of~$T'$ where the original splitting occurred, we get a graph~$\widetilde{T}$. Since $\widetilde{S}_1$ is a subdivision of~$S_1$ and $\widetilde{T}'$ is a subdivision of~$T'$, we have that $\widetilde{T}$ is a subdivision of~$T$. Since $a_1=1$, the starlike splitting of~$\widetilde{T}$ into $\widetilde{S}_1$ and~$\widetilde{T}'$ is regular if and only if the starlike splitting of~$T$ into $S_1$ and~$T'$ is regular. Since $\iota\bigl(\widetilde{T}'\bigr) = \iota(T')$, this means $\iota\bigl(\widetilde{T}\bigr) = \iota(T)$. By scaling the arithmetical structure on~$\widetilde{S}_1$ and merging it with the arithmetical structure on~$\widetilde{T}'$ as in Proposition~\ref{prop:mergingconstruction}, we get an arithmetical structure on~$\widetilde{T}$. By Theorem~\ref{thm:mergingcriticalgroup}, the critical group of this arithmetical structure on $\widetilde{T}$ is $\mathcal{G}_1 \oplus \mathcal{G}_2$, completing the proof in the case when $\beta=\iota(T)$.

Now suppose $\beta<\iota(T)$. If we subdivide all edges of~$T$ to obtain a tree $\overline{T}$, we will have $\iota\bigl(\overline{T}\bigr)=0$ because $\overline{T}$ has no adjacent vertices of degree at least~$3$. By Lemma~\ref{lem:2matchingsubdivision}, each of these subdivisions will decrease the splitting irregularity number by at most~$1$, so by doing only some subdivisions we can obtain a tree~$\widehat{T}$ with $\iota\bigl(\widehat{T}\bigr)=\beta$. Applying the above argument to~$\widehat{T}$ yields the result.
\end{proof}

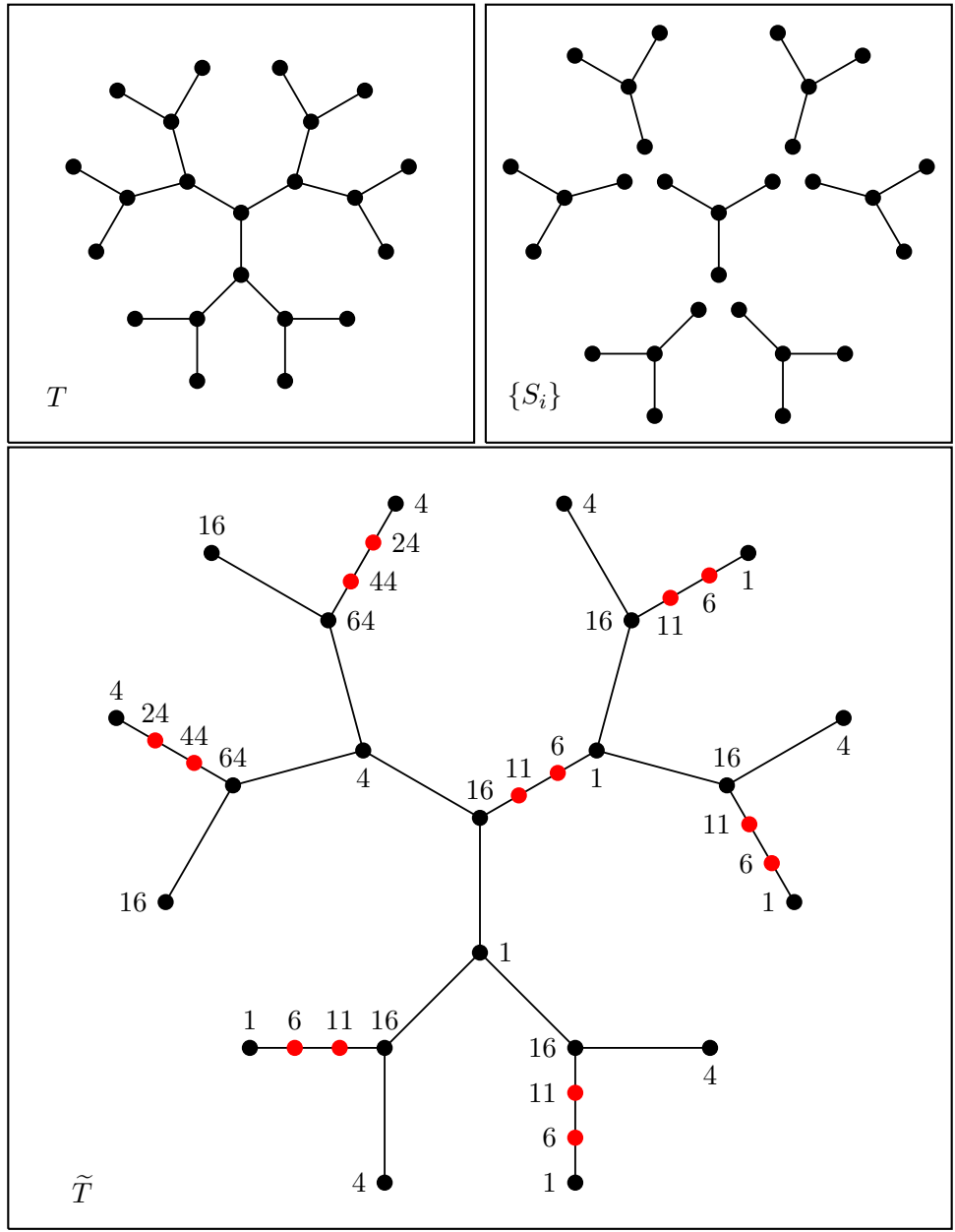
\begin{figure}
\centering
\begin{tikzpicture}[scale=.42]
\node at (-4.25,-4,5) {$T$};
\draw (-7.5,-7.4) -- (-7.5,6.7) -- (7.5,6.7) -- (7.5,-7.4) -- (-7.5,-7.4);
\node[main_node] (1) at (0,0) {};
\node[main_node] (2) at (30:2) {};
\node[main_node] (3) at (150:2) {};
\node[main_node] (4) at (270:2) {};
\node[main_node] (5) at ($(2)+(-15:2)$) {};
\node[main_node] (6) at ($(2)+(75:2)$) {};
\node[main_node] (7) at ($(3)+(195:2)$) {};
\node[main_node] (8) at ($(3)+(105:2)$) {};
\node[main_node] (9) at ($(4)+(225:2)$) {};
\node[main_node] (10) at ($(4)+(-45:2)$) {}; 
\node[main_node] (11) at ($(5)+(30:2)$) {};
\node[main_node] (12) at ($(5)+(-60:2)$) {};
\node[main_node] (13) at ($(6)+(30:2)$) {};
\node[main_node] (14) at ($(6)+(120:2)$) {};
\node[main_node] (15) at ($(8)+(60:2)$) {};
\node[main_node] (16) at ($(8)+(150:2)$) {};
\node[main_node] (17) at ($(7)+(150:2)$) {};
\node[main_node] (18) at ($(7)+(240:2)$) {}; 
\node[main_node] (19) at ($(10)+(0:2)$) {};
\node[main_node] (20) at ($(10)+(270:2)$) {}; 
\node[main_node] (21) at ($(9)+(180:2)$) {};
\node[main_node] (22) at ($(9)+(270:2)$) {};

\draw (1.center) -- (2.center);
\draw (1.center) -- (3.center);
\draw (1.center) -- (4.center);
\draw (2.center) -- (5.center);
\draw (2.center) -- (6.center);
\draw (3.center) -- (7.center);
\draw (3.center) -- (8.center);
\draw (4.center) -- (9.center);
\draw (4.center) -- (10.center);
\draw (5.center) -- (11.center);
\draw (5.center) -- (12.center);
\draw (6.center) -- (13.center);
\draw (6.center) -- (14.center);
\draw (8.center) -- (15.center);
\draw (8.center) -- (16.center);
\draw (7.center) -- (17.center);
\draw (7.center) -- (18.center);
\draw (10.center) -- (19.center);
\draw (10.center) -- (20.center);
\draw (9.center) -- (21.center);
\draw (9.center) -- (22.center);
\end{tikzpicture}
\begin{tikzpicture}[scale=.42]
\node at (-3.4,-4,6) {$\{S_i\}_{i=1}^7$};
\draw (-7.5,-7.4) -- (-7.5,6.7) -- (7.5,6.7) -- (7.5,-7.4) -- (-7.5,-7.4);
\node[main_node] (1) at (0,0) {};
\node (2) at (30:2.75) {};
\node[main_node] (2a) at ($(2)+(210:0.75)$) {};
\node[main_node] (2b) at ($(2)+(90:0.75)$) {};
\node[main_node] (2c) at ($(2)+(-30:0.75)$) {};
\node (3) at (150:2.75) {};
\node[main_node] (3a) at ($(3)+(210:0.75)$) {};
\node[main_node] (3b) at ($(3)+(90:0.75)$) {};
\node[main_node] (3c) at ($(3)+(-30:0.75)$) {};
\node (4) at (270:2.75) {};
\node[main_node] (4a) at ($(4)+(210:0.75)$) {};
\node[main_node] (4b) at ($(4)+(90:0.75)$) {};
\node[main_node] (4c) at ($(4)+(-30:0.75)$) {};
\node[main_node] (5) at ($(2c)+(-15:2)$) {};
\node[main_node] (6) at ($(2b)+(75:2)$) {};
\node[main_node] (7) at ($(3a)+(195:2)$) {};
\node[main_node] (8) at ($(3b)+(105:2)$) {};
\node[main_node] (9) at ($(4a)+(225:2)$) {};
\node[main_node] (10) at ($(4c)+(-45:2)$) {}; 
\node[main_node] (11) at ($(5)+(30:2)$) {};
\node[main_node] (12) at ($(5)+(-60:2)$) {};
\node[main_node] (13) at ($(6)+(30:2)$) {};
\node[main_node] (14) at ($(6)+(120:2)$) {};
\node[main_node] (15) at ($(8)+(60:2)$) {};
\node[main_node] (16) at ($(8)+(150:2)$) {};
\node[main_node] (17) at ($(7)+(150:2)$) {};
\node[main_node] (18) at ($(7)+(240:2)$) {}; 
\node[main_node] (19) at ($(10)+(0:2)$) {};
\node[main_node] (20) at ($(10)+(270:2)$) {}; 
\node[main_node] (21) at ($(9)+(180:2)$) {};
\node[main_node] (22) at ($(9)+(270:2)$) {};

\draw (1.center) -- (2a.center);
\draw (1.center) -- (3c.center);
\draw (1.center) -- (4b.center);
\draw (2c.center) -- (5.center);
\draw (2b.center) -- (6.center);
\draw (3a.center) -- (7.center);
\draw (3b.center) -- (8.center);
\draw (4a.center) -- (9.center);
\draw (4c.center) -- (10.center);
\draw (5.center) -- (11.center);
\draw (5.center) -- (12.center);
\draw (6.center) -- (13.center);
\draw (6.center) -- (14.center);
\draw (8.center) -- (15.center);
\draw (8.center) -- (16.center);
\draw (7.center) -- (17.center);
\draw (7.center) -- (18.center);
\draw (10.center) -- (19.center);
\draw (10.center) -- (20.center);
\draw (9.center) -- (21.center);
\draw (9.center) -- (22.center);
\end{tikzpicture}

\begin{tikzpicture}[scale=.911]
\node at (-6.25,-5.25) {$\widetilde{T}$};
\draw (-7,-6.1) -- (-7,5.5) -- (7,5.5) -- (7,-6.1) -- (-7,-6.1);
\node[main_node, label=above:{$16$}] (1) at (0,0) {};
\node[main_node, label=below:{$1$}] (2) at (30:2) {};
\node[main_node, label=below:{$4$}] (3) at (150:2) {};
\node[main_node, label=right:{$1$}] (4) at (270:2) {};
\node[main_node, label=above:{$16$}] (5) at ($(2)+(-15:2)$) {};
\node[main_node, label=left:{$16$}] (6) at ($(2)+(75:2)$) {};
\node[main_node, label=above:{$64$}] (7) at ($(3)+(195:2)$) {};
\node[main_node, label=right:{$64$}] (8) at ($(3)+(105:2)$) {};
\node[main_node, label=above:{$16$}] (9) at ($(4)+(225:2)$) {};
\node[main_node, label=left:{$16$}] (10) at ($(4)+(-45:2)$) {}; 
\node[main_node, label=below:{$4$}] (11) at ($(5)+(30:2)$) {};
\node[main_node, label=left:{$1$}] (12) at ($(5)+(-60:2)$) {};
\node[main_node, label=below:{$1$}] (13) at ($(6)+(30:2)$) {};
\node[main_node, label=right:{$4$}] (14) at ($(6)+(120:2)$) {};
\node[main_node, label=right:{$4$}] (15) at ($(8)+(60:2)$) {};
\node[main_node, label=above:{$16$}] (16) at ($(8)+(150:2)$) {};
\node[main_node, label=above:{$4$}] (17) at ($(7)+(150:2)$) {};
\node[main_node, label=left:{$16$}] (18) at ($(7)+(240:2)$) {}; 
\node[main_node, label=below:{$4$}] (19) at ($(10)+(0:2)$) {}; 
\node[main_node, label=left:{$1$}] (20) at ($(10)+(270:2)$) {}; 
\node[main_node, label=above:{$1$}] (21) at ($(9)+(180:2)$) {};
\node[main_node, label=left:{$4$}] (22) at ($(9)+(270:2)$) {};

\draw (1.center) -- (2.center);
\draw (1.center) -- (3.center);
\draw (1.center) -- (4.center);
\draw (2.center) -- (5.center);
\draw (2.center) -- (6.center);
\draw (3.center) -- (7.center);
\draw (3.center) -- (8.center);
\draw (4.center) -- (9.center);
\draw (4.center) -- (10.center);
\draw (5.center) -- (11.center);
\draw (5.center) -- (12.center);
\draw (6.center) -- (13.center);
\draw (6.center) -- (14.center);
\draw (8.center) -- (15.center);
\draw (8.center) -- (16.center);
\draw (7.center) -- (17.center);
\draw (7.center) -- (18.center);
\draw (10.center) -- (19.center);
\draw (10.center) -- (20.center);
\draw (9.center) -- (21.center);
\draw (9.center) -- (22.center);
\node[main_node, red, label=above:{$11$}] (2a) at (30:2/3) {};
\node[main_node, red, label=above:{$6$}] (2b) at (30:2*2/3) {};
\node[main_node, red, label=left:{$11$}] (12a) at ($(5)+(-60:2/3)$) {};
\node[main_node, red, label=left:{$6$}] (12b) at ($(5)+(-60:2*2/3)$) {};
\node[main_node, red, label=below:{$11$}] (13a) at ($(6)+(30:2/3)$) {};
\node[main_node, red, label=below:{$6$}] (13b) at ($(6)+(30:2*2/3)$) {};
\node[main_node, red, label=right:{$44$}] (15a) at ($(8)+(60:2/3)$) {};
\node[main_node, red, label=right:{$24$}] (15b) at ($(8)+(60:2*2/3)$) {};
\node[main_node, red, label=above:{$11$}] (21a) at ($(9)+(180:2/3)$) {};
\node[main_node, red, label=above:{$6$}] (21b) at ($(9)+(180:2*2/3)$) {};
\node[main_node, red, label=left:{$11$}] (18a) at ($(10)+(-90:2/3)$) {};
\node[main_node, red, label=left:{$6$}] (18b) at ($(10)+(-90:2*2/3)$) {}; 
\node[main_node, red, label=above:{$44$}] (17a) at ($(7)+(150:2/3)$) {};
\node[main_node, red, label=above:{$24$}] (17b) at ($(7)+(150:2*2/3)$) {};
\end{tikzpicture}
\caption{A tree~$T$, a starlike decomposition $\{S_i\}_{i=1}^7$ of~$T$, and the $\vr$-values of an arithmetical structure on~$\widetilde{T}$, a subdivision of~$T$, that has critical group $(\ZZ/4\ZZ)^7$.\label{fig:constructgroupsexample}}
\end{figure}

We also record a special case of Theorem~\ref{thm:constructgroups} that follows when $\beta=0$.
\begin{corollary}\label{cor:constructgroups}
Let $T$ be a finite tree. If $\mathcal{G}$ is a finite abelian group with at most $\ell(T)-2$ invariant factors, then there is an arithmetical structure on a subdivision of~$T$ whose critical group is~$\mathcal{G}$. 
\end{corollary}

The following example, shown in Figure~\ref{fig:constructgroupsexample}, illustrates Theorem~\ref{thm:constructgroups}.

\begin{example}
Let $T$ be the tree shown in the top left of Figure~\ref{fig:constructgroupsexample}, which has $\iota(T)=3$ and $\ell(T)=12$. By Theorem~\ref{thm:constructgroups} with $\beta=\iota(T)$, for any group with at most $7$~invariant factors, there is a subdivision~$\widetilde{T}$ of~$T$ with $\iota\bigl(\widetilde{T}\bigr)=\iota(T)$ and an arithmetical structure on~$\widetilde{T}$ that has this group as critical group. Let $\mathcal{G} = (\ZZ/4\ZZ)^7$. Following the proof of Theorem~\ref{thm:constructgroups}, we find a starlike decomposition $\{S_i\}_{i=1}^7$ of~$T$, shown in the top right of Figure~\ref{fig:constructgroupsexample}. In this particular example, every starlike decomposition will be the same, up to reordering. We take each~$S_i$, which in this example are star graphs with three leaves, and subdivide them to obtain graphs having arithmetical structures with critical group $\ZZ/4\ZZ$. In this case, each $3$-star becomes a broom graph with two prongs and six vertices, with a center $\vr$-value of~$16$, $\vr$-values of $4$ and~$1$ at the prongs, and $\vr$-values of $11$, $6$, and~$1$ along the remaining tentacle. Merging these gives the tree~$\widetilde{T}$ and arithmetical structure shown at the bottom of Figure~\ref{fig:constructgroupsexample}, with $\vr$-values scaled in accordance with Proposition~\ref{prop:mergingconstruction} when necessary. The resulting tree~$\widetilde{T}$ is a subdivision of~$T$ with $\iota\bigl(\widetilde{T}\bigr)=3$, and the resulting arithmetical structure has critical group $\mathcal{G}=(\ZZ/4\ZZ)^7$. 
\end{example}

In this example, and in general, one cannot subdivide arbitrary edges of the tree while preserving the splitting irregularity number, so it is not immediately clear where to subdivide without first finding a starlike decomposition of the tree.

\begin{remark}
The results of this section imply that for any finite abelian group~$\mathcal{G}$ there is an arithmetical structure on a tree with critical group~$\mathcal{G}$, but the same thing is also true for some other classes of graphs. Recall that a graph~$G$ is unicyclic if $\abs{E(G)} = \abs{V(G)}$, and more generally is $c$-cyclic if $\abs{E(G)} = \abs{V(G)}+c-1$. Since there is an arithmetical structure on a cyclic graph with trivial critical group and an $\vr$-value of~$1$ at one of its vertices, Proposition~\ref{prop:broomgraphs} and Theorem~\ref{thm:mergingcriticalgroup} allow us to construct an arithmetical structure on a $c$-cyclic graph whose critical group is~$\mathcal{G}$, for any positive integer~$c$.
\end{remark}

\section*{Acknowledgments}

We would like to thank Dino Lorenzini for helpful comments on a previous version of this paper and ICERM for their support through the Collaborate@ICERM program. Alexander Diaz-Lopez’s research is supported by National Science Foundation grant DMS-2211379. Joel Louwsma was partially supported by two Niagara University Summer Research Awards.

\bibliographystyle{amsplain}
\bibliography{ChipFiringTrees}

\end{document}